\documentclass[preprint,3p,12pt]{elsarticle}

\usepackage{amssymb}

\usepackage{graphicx}%
\usepackage{multirow}%
\usepackage{amsmath,amssymb,amsfonts}%
\usepackage{amsthm}%
\usepackage{mathrsfs}%
\usepackage[title]{appendix}%
\usepackage{xcolor}%
\usepackage{textcomp}%
\usepackage{manyfoot}%
\usepackage{booktabs}%
\usepackage{algorithm}%  
\usepackage{algorithmicx}%
\usepackage{algpseudocode}%
\usepackage{listings}%
\usepackage[hidelinks]{hyperref}
\usepackage{subfigure}
\usepackage[capitalize,noabbrev]{cleveref}
\usepackage{microtype}
\usepackage{microtype}

\theoremstyle{plain}
\newtheorem{theorem}{Theorem}[section]
\newtheorem{proposition}[theorem]{Proposition}
\newtheorem{lemma}[theorem]{Lemma}
\newtheorem{corollary}[theorem]{Corollary}
\theoremstyle{definition}

\theoremstyle{remark}

\usepackage{xparse}
\DeclareFontFamily{U}{ntxmia}{}
\DeclareFontShape{U}{ntxmia}{m}{it}{<-> ntxmia }{}
\DeclareFontShape{U}{ntxmia}{b}{it}{<-> ntxbmia }{}
\DeclareSymbolFont{lettersA}{U}{ntxmia}{m}{it}
\SetSymbolFont{lettersA}{bold}{U}{ntxmia}{b}{it}

\ExplSyntaxOn
\NewDocumentCommand{\varmathbb}{m}
 {
  \tl_map_inline:nn { #1 }
  {
    \use:c { varbb##1 }
  }
 }
\tl_map_inline:nn { ABCDEFGHIJKLMNOPQRSTUVWXYZ }
 {
  \exp_args:Nc \DeclareMathSymbol{varbb#1}{\mathord}{lettersA}{\int_eval:n { `#1+67 }}
 }
\exp_args:Nc \DeclareMathSymbol{varbbk}{\mathord}{lettersA}{169}
\ExplSyntaxOff

\newcommand{\bv}{\mathbf{v}}
\newcommand{\ve}{\mathcal{I}}
\newcommand{\RR}{\mathbb{R}}
\newcommand{\expec}{\mathbb{E}}

\newcommand{\mvar}{\mathrm{Var}_M}
\newcommand{\tr}{{\mathrm{tr}}} 

\newcommand{\prox}{\mathrm{Prox}}
\newcommand{\Proj}{{\mathrm{Proj}}}
\newcommand{\range}{\mathrm{range}\,} 
\DeclareMathOperator*{\argmin}{argmin}

\newcommand{\dist}{\mathrm{dist}}
\newcommand{\reals}{\mathbb{R}}

\newcommand{\opC}{{\varmathbb{C}}}

\newcommand{\opG}{{\varmathbb{G}}}
\newcommand{\opH}{{\varmathbb{H}}}
\newcommand{\opI}{{\varmathbb{I}}}

\newcommand{\opN}{{\varmathbb{N}}}

\newcommand{\opS}{{\varmathbb{S}}}
\newcommand{\opT}{{\varmathbb{T}}}

\newcommand{\vg}{{\mathbf{g}}}
\newcommand{\vu}{{\mathbf{u}}}
\newcommand{\vv}{{\mathbf{v}}}
\newcommand{\vw}{{\mathbf{w}}}
\newcommand{\vx}{{\mathbf{x}}}
\newcommand{\vy}{{\mathbf{y}}}

\newcommand{\vC}{{\mathbf{C}}}

\newcommand{\cF}{{\mathcal{F}}}
\newcommand{\cH}{{\mathcal{H}}}
\newcommand{\bigO}{\mathcal{O}}

\newcommand\norm[1]{\left\lVert #1 \right\rVert}
\newcommand\Mnorm[1]{\left\lVert #1 \right\rVert_M}

\journal{Journal of Mathematical Analysis and Applications}

\begin{document}

\begin{frontmatter}



\title{Coordinate-Update Algorithms can Efficiently Detect Infeasible Optimization Problems}


% \author{}
\author{Jinhee Paeng\fnref{label1}}
\ead{jhpaeng@stanford.edu}

\author{Jisun Park\fnref{label2}}
\ead{jisunpark@princeton.edu}

\author{Ernest K. Ryu\corref{cor1}\fnref{label3}}
\ead{eryu@math.ucla.edu}

\cortext[cor1]{Corresponding author.}

\affiliation[label1]{organization={Institute for Computational and Mathematical Engineering, Stanford University},%Department and Organization
            addressline={\\475 Via Ortega}, 
            city={Stanford},
            postcode={94305}, 
            state={CA},
            country={United States}}

\affiliation[label2]{organization={Department of Operations Research and Financial Engineering, Princeton University},%Department and Organization
            addressline={\\98 Charlton street}, 
            city={Princeton},
            postcode={08544}, 
            state={NJ},
            country={United States}}

\affiliation[label3]{organization={Department of Mathematics, University of California, Los Angeles},%Department and Organization
            addressline={\\520 Portola Plaza}, 
            city={Los Angeles},
            postcode={90095}, 
            state={CA},
            country={United States}}

\begin{abstract}
Coordinate update/descent algorithms are widely used in large-scale optimization due to their low per-iteration cost and scalability, but their behavior on infeasible or misspecified problems has not been much studied compared to the algorithms that use full updates. For coordinate-update methods to be as widely adopted to the extent so that they can be used as engines of general-purpose solvers, it is necessary to also understand their behavior under pathological problem instances. In this work, we show that the normalized iterates of randomized coordinate-update fixed-point iterations (RC-FPI) converge to the infimal displacement vector and use this result to design an efficient infeasibility detection method. We then extend the analysis to the setup where the coordinates are defined by non-orthonormal basis using the Friedrichs angle and then apply the machinery to decentralized optimization problems.
\end{abstract}


% \begin{highlights}
% \item The normalized iterate $x^k/k$ of {RC-FPI}  converges to the infimal displacement vector.
% \item The normalized iterate's asymptotic variance is bounded in $\bigO\left( 1/k\right)$.
% \item Infeasibility detection is done by hypothesis testing using the normalized iterate.
% \item Results are extended to non-orthonormal basis via Friedrichs Angle.
% \end{highlights}


\end{frontmatter}



\section{Introduction}\label{Introduction}
Coordinate update/descent algorithms are widely used in large-scale optimization due to their low per-iteration cost and scalability. These algorithms update only a single block of coordinates of an optimization variable per iteration in contrast to full or stochastic gradient algorithms, which update all variables every iteration. The convergence of coordinate update algorithms has been analyzed extensively, and they have been shown to achieve strong practical and theoretical performance in many large-scale machine learning and optimization problems \citep{nesterov2012efficiency} for non-pathological problem instances.

However, the behavior of coordinate update algorithms on infeasible or misspecified problems has not been analyzed, which sharply contrasts with algorithms that use full (deterministic) updates. The recent interest in building general-purpose optimization solvers with first-order algorithms has led to much work analyzing the behavior of full-update first-order algorithms on pathological problem instances so that the solvers can robustly detect such instances. For coordinate-update methods to be as widely adopted, to the extent that they can be used as engines of general-purpose solvers, it is necessary to also understand their behavior under pathological problem instances.

\subsection{Summary of results, contribution, and organization}
In this work, we analyze the behavior of randomized coordinate-update fixed-point iterations (RC-FPI) applied to inconsistent problem instances. Analogous to the classical results of the full-update fixed-point iterations, we show that the normalized iterate ${x^k}/{k}$ generated by RC-FPI converges toward the infimal displacement vector, which serves as a certificate of infeasibility, in the sense of both $L^2$ and almost sure convergence. We then bound the asymptotic bias and variance of the estimator, thereby establishing an asymptotic convergence rate. Finally, we extend the analysis to the setup where the coordinates are defined by non-orthonormal basis using the Friedrichs angle and then apply the machinery to decentralized optimization problems.

The paper is organized as follows. \cref{sec:pre} sets up notations and reviews known results and notions. \cref{sec:rcupdate} provides clear definition of randomized-coordinate update setting. \cref{sec:linearrate} presents the $L^2$ and almost sure convergence of the normalized iterate. \cref{sec:variance} provides the asymptotic upper bound for the normalized iterate. \cref{sec:infeas_detect} then uses these results to build the infeasibility detection in \eqref{RC-FPI}. \cref{sec:decent_opt} extends our result to the non-orthogonal basis, allowing application to the optimization methods such as \eqref{iter:PG-EXTRA}. \cref{sec:conclusion} concludes the paper.

\subsection{Prior work}

\subsubsection{FPI of inconsistent case.} 
Behavior of the inconsistent fixed-point iteration has been first characterized by \citet{Browder1966}, who showed that the iterates are not bounded.
Later, \citet{Pazy1971AsymptoticBO} showed that the iterates actually diverge in a sense that $\lim_{k\to\infty} \frac{\opT^k x^0}{k} = -\bv $, and this work also led to the similar results in more general Banach space settings \citep{reich1973asymptotic,Reich1981asymptotic,Reich1982asymptotic,PlantReich1983asymptotics} or geodesic spaces \citep{ariza2014firmly,Nicolae2013asymptotic}.
If the operator is more than just non-expansive, then the difference of iterates is also convergent to $\bv$; see \citet{bruck1977weak,bailion1978asymptotic,ReichSharfrir1987}.

There are also in-depth analyses on the characteristics of infimal displacement vector, regarding its direction
\citep{BauschkeDouglasMoursi2016,Ryu2018,Gutierrez2021comments} and the composition and convex combinations of non-expansive operators \citep{BauschkeMoursi2018magnitude,BauschkeMoursi2020minimal}.

\subsubsection{Infeasibility detection and numerical solvers.}
Fixed-point iteration covers a broad range of optimization algorithms, including Douglas-Rachford splitting (DRS) \citep{lions1979splitting} or alternating direction method of multipliers (ADMM) \citep{gabay1976dual,Glowinski1975admm}, which are commonly used as first-order methods for solving general convex optimization problems.
The infimal displacement vector of DRS and ADMM operator have been recently studied \citep{BauschkeHareMoursi2014,BauschkeDaoMoursi2016douglas,BauschkeMoursi2016douglas,BauschkeMoursi2020,BauschkeMoursi2021,Moursi2022douglas}, and it was proven to have meaning in terms of primal and dual problems as well \citep{BanjacLygeros2021asymptotic,Banjac2021}.
Related to such behaviors, ADMM-based infeasibility detecting algorithms have been suggested \citep{banjac2019infeasibility,liu2019new,ryu2019douglas}, which led to the first-order numerical solvers like OSQP \citep{osqp} and COSMO \citep{Garstka_2021}.
Apart from above, SCS \citep{scs,sopasakis2019superscs} uses homogeneous self-dual embedding \citep{o2016conic,o2021operator}.

% XXX
% Convergence of Gradient Descent on Separable Data
% XXX

\subsubsection{Randomized coordinate update and RC-FPI.} 

{
% \color{blue} 
% XXX Add reference of reviewer 2. XXX

The randomized coordinate update method is an iterative method that updates only a subset of randomly chosen coordinates or blocks of coordinates at each iteration.
Coordinate descent \cite{wright2015coordinate} is a coordinate-update gradient method \citep{hildreth1957quadratic,d1959convex,warga1963minimizing,garkavi1980method,luo1992convergence,grippo2000convergence,tseng2001convergence}, and such method is also popular in proximal setup \citep{auslender1992asymptotic,razaviyayn2013unified,xu2013block}, prox-linear setup \citep{tseng2009coordinate,yun2011coordinate,nesterov2012efficiency,beck2013convergence,shi2014sparse,xu2015alternating,zhou2016global,xu2017globally,hong2017iteration}, distributed (or asynchronous) optimization setup \citep{fercoq2014fast,liu2014asynchronous,liu2015asynchronous,peng2016arock}, and even in discrete optimization \citep{farsa2020discrete,hazimeh2020fast,jager2020blockwise}.
Furthermore, there are attempts to hybrid coordinate update with full update in primal-dual algorithms
\citep{ChambolleEhrhardtRichtarikSchonlieb2018_stochastic, FercoqBianchi2019_coordinatedescent}

Regarding how the coordinates and blocks to update are chosen, cyclic coordinate-update \cite{oswald2017random,lee2019random,wright2020analyzing,gurbuzbalaban2020randomness} has gained attention due to its good practical performance.
There are in-depth complexity analysis and accelerated variants of coordinate descent method as well \citep{lin2014accelerated,richtarik2014iteration,fercoq2015accelerated,allen2016even,qu2016coordinate,nesterov2017efficiency,hong2017iteration,hanzely2019accelerated,bertrand2021anderson,alacaoglu2022complexity,nutini2022let,cai2023cyclic}.

Randomized coordinate-update for fixed-point iteration has been first proposed by \citet{verkama1996random}.
General framework for randomized block-coordinate fixed-point iteration was suggested by \citet{combettes2015stochastic,combettes2019stochastic}, followed by similar line of works including block-coordinate update fixed-point iteration in asynchronous parallel setup \citep{peng2016arock}, forward-backward splitting \citep{chouzenoux2016block,salzo2022parallel}, Douglas-Rachford splitting \citep{briceno2019random}, and so on.
It also led to the refined analysis in cyclic fixed-point iterations \citep{chow2017cyclic,peng2019cyclic}, and the iteration complexity of coordinate update fixed-point iterations and their variants \citep{lu2015complexity,combettes2018linear,sun2021worst,tran2023accelerated}.

}
% \citep{ChambolleEhrhardtRichtarikSchonlieb2018_stochastic} : PDHG with stochastic dual update
% \citep{FercoqBianchi2019_coordinatedescent} : coordinate descent primal-dual with coordinate update in dual.

\subsubsection{Friedrichs angle and splitting methods.}

Friedrichs angle \citep{Friedrichs1937OnCI,deutsch1985rate,deutsch1995angle} measures an angle between a number of subspaces,
and is often used to characterize the convergence rate of projection methods.
\citep{aronszajn1950theory,kayalar1988error,bauschke1996projection,badea2010generalization,nishihara2014convergence,falt2017optimal,reich2017optimal,aragon2018new,aragon2019optimal,falt2021generalized}.
This kind of approach has been extended to cover splitting methods such as DRS and ADMM as well \citep{bauschke2014rate,raghunathan2014admm,bauschke2016optimal,liang2017local,banjac2018tight,behling2018circumcentering,liang2018local}.

\section{Preliminaries and notations} \label{sec:pre}
In this section, we set up notations and review known results. First, let's clarify the underlying space. Throughout this paper, a Hilbert space refers to a real Hilbert space. The underlying space is a real Hilbert space $\cH$, which is consisted of $m$ real Hilbert spaces.
\[
\cH = \cH_1 \oplus \cH_2 \oplus \dots \cH_m.
\]
An element $u \in \cH$ can be decomposed into $m$ blocks as
\[
u=\left( u_1, u_2, \dots , u_m\right), \quad u_i \in \cH_i,
\]
and $u_i$ is called the $i$th block coordinates of $u$.
  
The Hilbert space $\cH$ has its induced norm and inner product as
\[
\|x\|^2=\sum_{i=1}^m \|x_i\|_i^2,
\qquad
\langle x,y\rangle=\sum_{i=1}^m \langle x_i,y_i\rangle_i,
\]
for all $x, y \in \cH$, where $\| \cdot\|_i$ and $\langle \cdot, \cdot\rangle_i$ are the norm and inner product of $\cH_i$ and $x_i, y_i$ are $i$th block coordinates of $x, y$, respectively. 

Consider a linear, bounded, self-adjoint and positive definite operator $M:\cH \to \cH$. The $M$-norm and $M$-inner product of $\cH$ are defined as
\[
\|x\|_M=\sqrt{\langle x, Mx \rangle},
\qquad
\langle x,y\rangle_M=\langle x, My \rangle,
\]
which can also be a pair of norm and inner product of the space $\cH$. $\| \cdot\|$ and $\langle \cdot, \cdot\rangle$ are simply the instances of $M$-norm and $M$-inner product with $M$ as an identity map. For the remark, the map $M$ can be expressed as a symmetric positive definite matrix if $\cH = \RR^n$. In this case, $M$-inner product and $M$-norm are 
\[
\|x\|_M=\sqrt{x^T M x},
\qquad
\langle x,y\rangle_M=x^T M y.
\]

Define the $M$-variance of a random variable $X$ with the domain $\cH$ as
\[
    \mvar [X] = \expec [\|X\|^2_M] - \| \expec [X] \|^2_M.
\]
We develop the theory of Sections~\ref{sec:linearrate} and \ref{sec:variance} with the general $M$-norm so that the theory is applicable to the applications of Section~\ref{sec:decent_opt}.

\subsection{Operators}
Denote $\opI\colon\cH\to\cH$ as the identity operator.
For an operator $\opT\colon\cH\to\cH$, let $\range \opT$ be a {range} of $\opT$.
If $x_\star\in\cH$ is a point such that $x_\star = \opT x_\star$, it is called a fixed point of $\opT$.

We say an operator $\opT\colon\cH \to \cH$ is {non-expansive} with respect to $\Mnorm{\cdot}$ if
\[
\Mnorm{\opT x - \opT y} \leq \Mnorm{x-y}, \qquad \forall x, y \in \cH,
\]
and is $\theta$-{averaged} for $\theta\in(0,1)$ if $\opT = (1-\theta) \opI + \theta \opC$ for some non-expansive operator $\opC$.
For notational convenience, we will refer to non-expansive operators as $\theta$-{averaged} operators with $\theta=1$, even though, strictly speaking, $\theta=1$ means the operator is not averaged.
An operator $\opS\colon\cH\to\cH$ is $({1}/{2})$-{cocoercive} with respect to ${\Mnorm{\cdot}}$ if
\[
\left< \opS x - \opS y,\, x - y \right>_M\geq \frac{1}{2} \Mnorm{\opS x - \opS y}^2, \qquad \forall x, y \in \cH.
\]

$\opT$ is non-expansive if and only if $\opS = \opI - \opT$ is $({1}/{2})$-{cocoercive}.
Also, $\opT$ is $\theta$-averaged for some $\theta\in(0,1)$ if and only if $\opS = \theta^{-1}{(\opI-\opT)}$ is $({1}/{2})$-{cocoercive}.

% \textbf{Fixed-point iteration.}

% \textbf{Variance.} \textcolor{red}{(XXX definition of variance at Appendix? XXX)}
% For a random variable $X$ defined in $\cH$, we will use the term variance of $X$ as a trace of the covariance matrix of $X$. The variance of $X$, or  $\mvar \left( X\right)$, is computed as 
% \[
% \mvar\left( X\right) = \expec \left[ \Mnorm{X}^2 \right] - \Mnorm{ \expec \left[ X \right]}^2.
% \]

\subsection{Inconsistent operators and infimal displacement vector}
We say an operator $\opT\colon\cH\to\cH$ is consistent if it has a fixed point, and inconsistent if it does not have a fixed point.
$\opT$ is consistent if and only if $0\in\range(\opI-\opT)$.
When $\opT$ is non-expansive, the closure $\overline{\range \left(\opI-\opT\right)}$ is a nonempty closed convex set, so it has a unique minimum-norm element, which we denote by $\bv$ \citep{Pazy1971AsymptoticBO}.

We call $\bv$ the  \emph{infimal displacement vector} of $\opT$ \citep{BauschkeBorweinLewis1997,BauschkeHareMoursi2014}.
Alternatively, $\bv$ is the projection of $0$ onto $\overline{\range \left(\opI-\opT\right)}$. Equivalently, $\bv \in \overline{\range \left(\opI-\opT\right)}$ is the infimal displacement vector of $\opT$ if and only if
\begin{equation}
    \left< y - \bv,\, \bv \right>_M\geq 0, \qquad \forall y \in \range(\opI-\opT).
    \label{eq:inf-char}
\end{equation}
% with $\overline{\range(\opI-\opT)}$ being a closure of the range of $\opI-\opT$.

For a convex optimization problem, let $\opT$ be an operator corresponding to an iterative first-order method, such as the  Douglas-Rachford splitting (DRS) operator \citep{lions1979splitting}, and let $\bv$ be its infimal displacement vector. Loosely speaking, if the optimization problem is feasible and the problem is well-behaved, then $\bv=0$. (However, it is possible for ``weakly infeasible'' problems to have $\bv=0$, so $\bv=0$ does not guarantee feasibility.) On the other hand, $\bv\neq0$ implies that the problem or its dual problem is infeasible, so $\bv\ne 0$ serves as a certificate of infeasibility \citep{liu2019new,banjac2019infeasibility}.

\subsection{Fixed point iteration and normalized iterate}
The fixed-point iteration \eqref{FPI} with respect to an operator $\opT\colon\cH\to \cH$ is defined as
\begin{equation} \label{FPI} \tag{FPI}
    x^{k+1} = \opT x^{k}, \quad k=0,1,2,\dots,
\end{equation}
where $x^0\in\cH$ is a starting point.

Let $x^0,x^1,x^2,\ldots$ be the iterates of \eqref{FPI}.
We call ${x^k}/{k}$ the $k$th \emph{normalized iterate} of \eqref{FPI} for $k=1,2,\dots$. The seminal work of \citet{Pazy1971AsymptoticBO} characterizes the dynamics of normalized iterates of \eqref{FPI}.
\begin{theorem}[\citet{Pazy1971AsymptoticBO}]
\label{thm:pazy}
Let $\opT\colon\cH\to\cH$ be non-expansive.
Let $x^0,x^1,x^2,\ldots$ be the iterates of \eqref{FPI}.
    Then, the normalized iterate ${x^k}/{k}$ converges strongly,
    \begin{align*}
        \frac{x^k}{k} \rightarrow  -\bv
    \end{align*}
    as $k\to\infty$, 
    where $\bv$ is the infimal displacement vector of $\opT$.
\end{theorem}

Since the underlying space is a Hilbert space, we clarify that the convergence in the space $\cH$ throughout this paper refers to the strong convergence of Hilbert space.

% Given an optimization problem, let $\opT$ be an operator whose fixed-point iteration represents iterative first-order method solving such problems.
% The fixed-point of $\opT$ represents the solution for both primal and dual solution, so $\opT$ is consistent if and only if the strong duality holds.
% This problem has an optimal solution if and only if $\opT$ is consistent.
% If the problem is either primal infeasible or dual infeasible, $\opT$ is inconsistent, and our main focus will be such cases: the given primal or dual problem is infeasible, i.e., there is no point satisfying the problem constraint.

\section{Randomized-coordinate update setup}
\label{sec:rcupdate}

In this section, we focus on a variant of \eqref{FPI} with randomized coordinate updates. 
Consider a {$\theta$-averaged} operator $\opT\colon\cH\to \cH$ with its corresponding {$({1}/{2})$-{cocoercive}} operator $\opS=\theta^{-1}{(\opI-\opT)}$ with $\theta\in(0,1]$.
To clarify, we will refer to non-expansive operators as $\theta$-{averaged} operators with $\theta=1$.
Define $\opS_i\colon\cH\to \cH$ for $i=1, 2, \dots, m$ as $\opS_i x=(0,\dots,0,(\opS x)_i,0,\dots,0)$, where $(\opS x)_i \in \cH_i$.
% For $i=1, 2, \dots, m$, define $\opS_i\colon\cH\to \cH$ so that $\opS_i x$ is nonzero only on the $i$th block with the $i$th block equal to $(\opS x)_i$, i.e., $\opS_i x=(0,\dots,0,(\opS x)_i,0,\dots,0)$.
% % Therefore, 
% % \[
% %     \opS x = \sum_{i=1}^m \opS_i x,\qquad\forall\,x\in \mathbb{R}^n.
% % \]

We call $\ve=\left(\ve_1, \ve_2, \dots , \ve_m \right) \in \left[ 0,1\right]^m\subset \RR^m$ a \emph{selection vector} and use it as follows.
Define $\opS_\ve\colon\cH\to\cH$ and $\opT_\ve\colon\cH\to\cH$ as 
\[
\opS_\ve = \sum_{i=1}^{m} \ve_i \opS_i ,
\qquad
\opT_\ve = \opI - \theta \opS_\ve.
\]
We can think of $\opS_\ve$ as the selection of blocks based on $\ve$ and $\opT_\ve$ as the update based on the selected blocks.
% When the selection vector is a random variable following some probability distribution on $\left[ 0,1\right]^m$, we say $\opT_\ve$ is a randomized coordinate operator.
The randomized coordinate fixed-point iteration \eqref{RC-FPI} is defined as
\begin{equation} \label{RC-FPI} \tag{RC-FPI}
    x^{k+1} = \opT_{\ve^k} x^k, \qquad k=0,1,2,\dots,
\end{equation}
where $\ve^0, \ve^1, \ldots$ is sampled IID and $x^0\in\cH$ is a starting point.
\eqref{RC-FPI} is a randomized variant of \eqref{FPI}.
% From the condition \eqref{condition:P}, we view $\alpha$ as the expected stepsize, and all block coordinates share the same stepsize $\alpha$ in expectation. 

{
% \color{blue}
Throughout this paper, we assume that $\ve$ is randomly sampled from a distribution on $\left[ 0,1\right]^m$ that satisfies the \emph{uniform expected step-size condition}
\begin{equation} \label{condition:P} 
    \expec_{\ve  } \left[ \ve \right] =\alpha\mathbf{1}
\end{equation}
for some $\alpha \in (0,1]$, where $\mathbf{1}\in \mathbb{R}^m$ is the vector with all entries equal to $1$.
(Note, $\ve \in \left[ 0,1\right]^m$ already implies $\alpha\in [0,1]$. We are additionally assuming that $\alpha>0$.)
}

The uniform expected step-size condition \eqref{condition:P} allows us to view one step of \eqref{RC-FPI} to be corresponding to a step of \eqref{FPI} with $\Bar{\opT}\colon \cH\to\cH$  defined as
\[
\Bar{\opT}x= \expec_{\ve } \left[ \opT_\ve x \right], \qquad \forall x \in \cH.
\]
Equivalently, $\Bar{\opT} =\opI - \alpha \theta \opS$.

{
% \color{blue}
For any $u\in \cH$ and selection vector $\ve$, define
\[
u_\ve =
\sum_{i=1}^m 
\underbrace{\ve_i }_{\in \mathbb{R}}
\underbrace{
(0,\dots,0,u_i,0,\dots,0)}_{\in \cH},
\]
where $u_i\in \cH_i$ for $i=1,\dots,m$.
If $\ve$ satisfies the uniform expected step-size condition \eqref{condition:P} with $\alpha\in (0,1]$, then clearly $\expec_{\ve} \left[  u_\ve \right] = \alpha u$.
Let $\beta>0$ be a coefficient such that
\begin{equation}
 \expec_{\ve} \left[  \Mnorm{u_\ve}^2 \right] \leq \beta \Mnorm{u}^2,
 \qquad\forall\,u\in \cH.
 \label{eq:beta-def}
\end{equation}

\begin{lemma} \label{lem:expec_norm}
% Let's consider the norm of $\cH$ as $\| \cdot\|$-norm.
Consider a Hilbert space $\cH$ with its norm $\|\cdot\|$.
If $\ve$ satisfies the uniform expected step-size condition \eqref{condition:P} with $\alpha\in (0,1]$, then $\beta=\alpha$ satisfies \eqref{eq:beta-def}.
\end{lemma}

\begin{proof}
Since $\ve_i \in \left[0,1 \right]$,
\begin{align*}
     \expec_{\ve  } \left[ \norm{ \sum_{i=1}^m \ve_i (0,\dots,0,u_i,0,\dots,0)}^2 \right] 
 &= \expec_{\ve  } \left[ \sum_{i=1}^m  \norm{\ve_i u_i}^2_i \right] \\
 & \leq \expec_{\ve  } \left[ \sum_{i=1}^m \ve_i \norm{ u_i}^2_i \right]
 =  \sum_{i=1}^m \alpha \norm{ u_i}^2 = \alpha \norm{u}^2.
\end{align*}
Thus, choose $\alpha=\beta$ to satisfy the inequality.
\end{proof}

As one can see from \cref{eq:beta-def}, choosing large $\beta$ is not a problem. The question is to choose the smallest $\beta$ for a given setting. Note that it is impossible to choose $\beta$ smaller than $\alpha^2$ since  $\beta \geq \alpha^2$ due to the Jensen's inequality. 

The assumption of $\beta \leq \alpha/\theta$ or $\beta <\alpha/\theta$ will be frequently required in this paper. Such assumptions are satisfied when the setting allows to choose small enough $\beta$. In the setting when the given space uses the Euclidean norm $\norm{\cdot}$, the \cref{lem:expec_norm} guarantees that $\beta$ can be chosen as $\alpha$. This directly guarantees $\beta \leq \alpha /\theta$, and $\beta < \alpha \theta $ if $\theta<1$. 

For the space with general $M$-norm, the smallest choice of $\beta$ gets inflated as the block coordinates are no longer orthogonal. Later on, the \cref{lem:f.1} validates the existence of  $\beta$ that satisfies $\beta \leq \alpha /\theta$ with the restricted Friedrichs angle.
Furthermore, smaller $\theta$ allows larger $\beta$ to satisfy $\beta \leq \alpha / \theta$. Since $\alpha$ and $\beta$ only depends on the setting of the underlying space and the distribution of $\ve$, user may further average $\opT$ until $\beta \leq \alpha / \theta$ is satisfied. 
}

\section{Convergence of normalized iterates}
\label{sec:linearrate}
% To establish the context, first consider the \eqref{FPI}.

In this section, we show that \eqref{RC-FPI} exhibits behavior similar to \cref{thm:pazy}.
Let $x^0,x^1,x^2,\ldots$ be the iterates of \eqref{RC-FPI}.
For each $k\in \mathbb{N}$, we likewise call ${x^k}/{k}$ the $k$th \emph{normalized iterate} of \eqref{RC-FPI} for $k=1,2,\dots$. Then, the normalized iterate converges to $-\alpha \bv$  both in $L^2$ and almost surely.
\[
\frac{ x^k} {k} \stackrel{L^2}{\rightarrow} -\alpha \bv, \quad \frac{ x^k} {k} \stackrel{\mathrm{a.s.}}{\rightarrow} -\alpha \bv.
\]
To clarify,  $\stackrel{L^2}{\rightarrow}$ and 
$\stackrel{\mathrm{a.s.}}{\rightarrow}$ respectively denote $L^2$ and almost sure convergence of random variables. The almost sure convergence means that ${x^k}/{k}$ being strongly convergent to $-\alpha \bv$ happens with the probability $1$.

\subsection{Properties of expectation on {RC-FPI}} \label{subsec:props}
We first characterize certain aspects of \eqref{RC-FPI} before establishing our main results.
First, we present a lemma exhibiting a non-expansiveness.
\begin{lemma}
\label{lem:reduce_expec}
Let $\opT\colon\cH\to\cH$ be $\theta$-averaged with respect to $\Mnorm{\cdot}$ with $\theta\in (0,1]$. Let $\ve$ be a random selection vector with distribution satisfying the uniform expected step-size condition \eqref{condition:P} with $\alpha\in (0,1]$. Assume \eqref{eq:beta-def} holds with some $\beta$ such that $\beta\leq\alpha/\theta$. Let $X$ and $Y$ be random variables on $\cH$ that are independent of $\ve$.
(However, $X$ and $Y$ need not be independent.) Then,
    \[
\mathop{\expec}_{\ve,X,Y} \left[ \Mnorm{\opT_\ve X - \opT_\ve Y}^2 \right]
\leq \mathop{\expec}_{X,Y} \left[ \Mnorm{ X - Y}^2 \right].
    \]
    
\end{lemma}

\begin{lemma}
\label{lem:L2_bound}
Let $\opT\colon\cH\to\cH$ be {$\theta$-averaged} respect to $\Mnorm{\cdot}$ with $\theta\in (0,1]$. Let $\ve$ be a random selection vector with distribution satisfying the uniform expected step-size condition \eqref{condition:P} with $\alpha\in (0,1]$. Assume \eqref{eq:beta-def} holds with some $\beta$.
For any $x, z\in\cH$, 
    \begin{align*}
            \mathop{\expec}_{\ve}\left[ \Mnorm{\opT_{\ve} x - \Bar{\opT}z }^2\right] 
            &\leq \Mnorm{x-z}^2
            + \theta^2 \left( \beta - \alpha^2\right) \Mnorm{\opS x}^2 
            -\alpha\theta \left( 1-\alpha \theta \right) \Mnorm{\opS x - \opS z}^2.
    \end{align*}
\end{lemma}
Proofs of \cref{lem:reduce_expec} and \cref{lem:L2_bound} are presented in the \ref{appendix:4.1}.

\subsection{$L^2$ convergence of normalized iterate} \label{subsec:L2}

\begin{theorem}
\label{thm:L2}
Let $\opT\colon\cH\to\cH$ be $\theta$-averaged with respect to $\| \cdot\|$-norm with $\theta\in \left( 0,1 \right]$.
Assume $\ve^0, \ve^1, \ldots$ is sampled IID from a distribution satisfying the uniform expected step-size condition \eqref{condition:P} with $\alpha\in (0,1]$.
Let $x^0,x^1,x^2,\ldots$ be the iterates of \eqref{RC-FPI}.
    Then
    \[
    \frac{ x^k} {k} \stackrel{L^2}{\rightarrow} -\alpha \bv
    \]
    as $k\to\infty$, where $\bv$ is the infimal displacement vector of $\opT$.
\end{theorem}

Before presenting the full proof, here is the key outline for the proof. 
Define another sequence $z^0,z^1,z^2,\dots$ with
\begin{equation} \label{FPI_T} \tag{FPI with $\Bar{\opT}$}
    z^{k+1} = \Bar{\opT} z^k, \qquad k=0,1,2,\dots.
\end{equation}
Apply \cref{lem:L2_bound} on the iterates of \eqref{RC-FPI} starting from $x^0$ and the iterates of \eqref{FPI_T} starting from $z^0=x^0$.
In \cref{appendix:L2}, we obtain a bound on the last two terms of \cref{lem:L2_bound} that is independent of $k$.
% Thus, by taking full expectation, for all $k\in \mathbb{N}$,
More specifically, for all $k=1,2,\dots$,
\[
\expec\left[ \Mnorm{ x^k - z^k }^2\right] 
\leq \expec\left[\Mnorm{x^{k-1}-z^{k-1}}^2 \right]+A,
\]
where $A =\left( 1-\alpha \theta \right) \left[ 2\sqrt{\alpha\theta}  \Mnorm{\opS x^0}^2
- \frac{\alpha}{\theta}  \Mnorm{\bv}^2 \right]$.
 Dividing by $k^2$ to get
\[
\expec\left[ \Mnorm{ \frac{ x^k} {k} - \frac{ z^k} {k} }^2 \right] \leq \frac{A}{k}
\]
and appealing to \cref{thm:pazy}, we conclude with the $L^2$ convergence.
We defer the detailed proof to \cref{appendix:L2}.

% Bound the last two terms in \cref{lem:L2_bound} using,
% \[
% \Mnorm{ \expec \left[ \opS \opT_{\ve^{k-1}}  \dots \opT_{\ve^0} x^0 \right]} \leq \beta^{1/2} \alpha^{-1} \Mnorm{\opS x^0},
% \]
% as a consequence of \cref{lem:reduce_expec}, and
% \[
% \Mnorm{\opS  z^k} \leq \Mnorm{\opS z^ {k-1}} \leq \dots \leq \Mnorm{\opS z^0},
% \]
% due to the {$\frac{1}{2}$-cocoercivity} of $\opS$.

\subsection{Proof of \cref{thm:L2}}
\label{appendix:L2}

% In the proof, the norm and inner product being $\| \cdot\|$ and $\langle \cdot, \cdot\rangle$ is not used until the final part.
In the proof, the norm $\|\cdot\|$ and the inner product $\langle \cdot, \cdot \rangle$ are not used until the final part.
Lemmas prior to the main proof of \cref{thm:L2} uses the general $M$-norm and $M$-inner product.

% We start the proof of \cref{thm:L2} by presenting two lemmas to make upper bounds for terms $\Mnorm{\opS z^k}$ and $\Mnorm{\expec \left[ \opS x^k\right] }$.
We start the proof of \cref{thm:L2} by presenting two lemmas to upper bound the terms $\Mnorm{\opS z^k}$ and $\Mnorm{\expec \left[ \opS x^k\right] }$.
{
% \color{blue}
Proofs of Lemmas~\ref{lem:bound_FPI} and \ref{lem:bound_expec} are deferred to \ref{appendix:4.2}.
}

\begin{lemma} \label{lem:bound_FPI}
    $\opT\colon\cH\to\cH$ is a {$\theta$-averaged} with $\theta \in \left( 0,1\right]$ and choose any starting point $z^0\in\cH$ for \eqref{FPI} with $\opT$. When $\opS = \theta^{-1}{(\opI-\opT)}$,
    \[
    \Mnorm{\opS  z^k} \leq \Mnorm{\opS z^ {k-1}} \leq \dots \leq \Mnorm{\opS z^0}.
    \]
\end{lemma}

For the remark, same inequality holds for the \eqref{FPI_T}, since $\opS= \theta^{-1}\alpha^{-1}(\opI - \bar{\opT})$.

\begin{lemma}
\label{lem:bound_expec}
    $\opT\colon\cH\to\cH$ is a {$\theta$-averaged} with $\theta \in \left( 0,1 \right]$ and choose any starting point $x^0\in\cH$ for \eqref{RC-FPI}. When $\opS = \theta^{-1}{(\opI-\opT)}$,
    \[
\Mnorm{ \expec \left[ \opS \opT_{\ve^k}  \dots \opT_{\ve^0} x^0 \right]} \leq \beta^{1/2} \alpha^{-1} \Mnorm{\opS x^0}
    \]
    holds if $\ve^0, \ve^1, \dots , \ve^k$ follow IID distribution with the condition \eqref{condition:P} with $\alpha \in \left( 0,1\right]$ and \eqref{eq:beta-def} holds with some $\beta$ such that $\beta\leq\alpha/\theta$.
\end{lemma}

\begin{proposition}
\label{lem:L2}
    Let $\opT\colon\cH\to\cH$ be $\theta$-averaged with respect to $\Mnorm{\cdot}$ with $\theta\in(0,1]$, and let $x^0, x^1, x^2, \dots$ be the iterates of \eqref{RC-FPI} and let $z^0, z^1, z^2, \dots$ be the iterates of \eqref{FPI_T}. Assume that the distribution of $\ve$ satisfies the uniform expected step-size condition \eqref{condition:P} with $\alpha\in (0,1]$ and \eqref{eq:beta-def} holds with some $\beta$ such that $\beta\leq\alpha/\theta$.
    % Given a $\theta$-averaged operator $\opT\colon\cH\to\cH$ with $\theta\in \left( 0,1 \right)$ and any starting point $x^0\in\cH$,
    % perform \eqref{RC-FPI} by $\opT$ to generate a random sequence $x^0, x^1, x^2, \dots$, and \eqref{FPI_T} to generate a sequence $z^0, z^1, z^2, \dots$ with $z^0 = x^0$.
    % $\opT\colon\cH\to\cH$ is a {$\theta$-averaged} and choose any starting point $x^0\in\cH$.
    % Perform {RC-FPI} by $\opT$ to generate a random sequence $x^0, x^1, x^2, \dots$, and {FPI} by $\Bar{\opT}$ to generate a sequence $z^0, z^1, z^2, \dots$ with $z^0 = x^0$.
    Then
    \begin{equation*}
        \begin{split}
&\expec\left[ \Mnorm{\frac{ x^k} {k} - \frac{ z^k} {k} }^2 \right] \leq \frac{1}{k^2}\Mnorm{x^0 - z^0}^2\\
&\hspace{0.2cm}+ \frac{1} {k} \left( 1-\alpha \theta \right) \left[ 2\sqrt{\alpha\theta}  \Mnorm{\opS x^0} \Mnorm{\opS z^0}
- \frac{\alpha}{\theta}  \Mnorm{\bv}^2 \right], \\
        \end{split}
    \end{equation*}
    where $\bv$ is the infimal displacement vector of $\opT$.
\end{proposition}

\begin{proof}
    The key step of proving \cref{lem:L2} is to bound the last two terms in \cref{lem:L2_bound}. To achieve this, rewrite the terms as
    \begin{equation*}
        \begin{split}
            &-\alpha\theta \left( 1-\alpha \theta \right) \Mnorm{\opS x - \opS z}^2 
            + \theta^2 \left( \beta - \alpha^2\right) \Mnorm{\opS x}^2 \\
            &\hspace{0.5cm} = - \theta \left(\alpha - \beta \theta  \right)\Mnorm{\opS x}^2
            +2\alpha\theta \left( 1-\alpha \theta \right) \left< \opS x , \opS z \right>_M
            -\alpha\theta \left( 1-\alpha \theta \right) \Mnorm{ \opS z}^2 \\
            &\hspace{0.5cm} \leq - \theta^{-1} \alpha \left( 1- \alpha \theta \right)\Mnorm{\bv}^2
            +2\alpha\theta \left( 1-\alpha \theta \right) \left< \opS x , \opS z \right>_M,
        \end{split}
    \end{equation*}
    where the last inequality is from $\bv$ being infimal displacement vector, i.e. $\Mnorm{\bv} \leq \Mnorm{\theta \opS x}, \Mnorm{\theta \opS z}$.

    Now use \cref{lem:L2_bound} with $x$ as $ x^k$ and $z$ as $ z^k$. Take a full expectation among $\ve^0, \ve^1, \dots , \ve^{k-1}$, then we get
    \begin{equation*}
        \begin{split}
            &\expec\left[ \Mnorm{ x^k -  z^k }^2 \right]\\
            &\hspace{0.5cm} \leq \expec\left[ \Mnorm{x^ {k-1}-z^ {k-1}}^2  \right]
            - \theta^{-1} \alpha \left( 1- \alpha \theta \right) \Mnorm{\bv}^2
            +2\alpha\theta \left( 1-\alpha \theta \right) \expec \left[ \left< \opS x^ {k-1} , \opS z^ {k-1} \right>_M\right] \\
            &\hspace{0.5cm} \leq \expec\left[ \Mnorm{x^ {k-1}-z^ {k-1}}^2  \right]
            - \theta^{-1} \alpha \left( 1- \alpha \theta \right) \Mnorm{\bv}^2
            +2\alpha\theta \left( 1-\alpha \theta \right) \Mnorm{\expec \left[ \opS x^ {k-1} \right] } \Mnorm{\opS z^ {k-1} } \\
            &\hspace{0.5cm} \leq \expec\left[ \Mnorm{x^ {k-1}-z^ {k-1}}^2  \right]
            - \theta^{-1} \alpha \left( 1- \alpha \theta \right) \Mnorm{\bv}^2
            +2\beta^{1/2}\theta \left( 1-\alpha \theta \right) \Mnorm{ \opS x^{0}  } \Mnorm{\opS z^{0} } \\
            &\hspace{0.5cm} \leq \expec\left[ \Mnorm{x^ {k-1}-z^ {k-1}}^2  \right]
            - \theta^{-1} \alpha \left( 1- \alpha \theta \right) \Mnorm{\bv}^2
            +2\sqrt{\alpha \theta} \left( 1-\alpha \theta \right) \Mnorm{ \opS x^{0}  } \Mnorm{\opS z^{0} },
        \end{split}
    \end{equation*}
    where the third inequality uses \cref{lem:bound_FPI} and \cref{lem:bound_expec}. Note that the term 
    \[
- \theta^{-1} \alpha \left( 1- \alpha \theta \right) \Mnorm{\bv}^2
            +2\sqrt{\alpha \theta} \left( 1-\alpha \theta \right) \Mnorm{ \opS x^{0}  } \Mnorm{\opS z^{0} }
    \]
    is independent from the random process and iterations. Thus, above inequality can be applied through iterations, resulting in
    \[
\expec\left[ \Mnorm{ x^k -  z^k }^2 \right]
\leq \Mnorm{x^0 - z^0 }^2 
+ k \left( - \frac{\alpha}{\theta} \left(1 - \alpha \theta \right)\Mnorm{\bv}^2
            +2\sqrt{\alpha \theta} \left( 1-\alpha \theta \right) \Mnorm{ \opS x^{0}  } \Mnorm{\opS z^{0} } \right) .
    \]
    Divide both sides by $k^2$ to obtain the desired result
    \begin{align*}
        &\expec\left[ \Mnorm{\frac{ x^k} {k} - \frac{ z^k} {k} }^2 \right] \\
&\hspace{0.5cm}\leq 
\frac{1} {k} \left( 
2\sqrt{\alpha \theta}  \left( 1-\alpha \theta \right) \Mnorm{ \opS x^{0}  } \Mnorm{\opS z^0}
- \frac{\alpha}{\theta} \left(1 - \alpha \theta \right)\Mnorm{\bv}^2
\right) +\frac{1}{k^2}\Mnorm{x^0 - z^0}^2.
    \end{align*}

\end{proof}

\begin{proof} [Proof of \cref{thm:L2}]
Since the $M=\opI$, we have $\beta = \alpha \leq \alpha/\theta$ due to \cref{lem:expec_norm}. Thus, we may apply \cref{lem:L2} with $z^0 = x^0$.
\[
\expec\left[ \norm{\frac{ x^k} {k} - \frac{ z^k} {k} }^2 \right]
\leq 
\frac{1} {k} \left( 
2\sqrt{\alpha \theta}  \left( 1-\alpha \theta \right) \norm{ \opS x^{0}  }^2
- \frac{\alpha}{\theta} \left(1 - \alpha \theta \right)\norm{\bv}^2
\right) .
    \]
When the limit $k\to \infty$ is taken, 
    \[
\lim_{k\to \infty} \expec\left[ \norm{\frac{ x^k} {k} - \frac{ z^k} {k} }^2 \right]=0, \quad \lim_{k\to\infty} \norm{\frac{ z^k} {k} +\alpha \bv } = 0,
    \]
    where the second equation is from \cref{thm:pazy}. These two limits provide $L^2$ convergence of normalized iterate, namely
    \[
 \frac{ x^k} {k}  \stackrel{L^2}{\rightarrow} -\alpha \bv,
    \]
    as $k\to \infty$.
\end{proof}

\subsection{Almost sure convergence of normalized iterate} \label{subsec:as}
\begin{theorem}
\label{thm:as}
%     $\opT\colon\cH\to\cH$ is a {$\theta$-averaged} with $\theta \in \left( 0,1 \right)$ and choose any starting point $x^0\in\cH$. Perform {RC-FPI} by $\opT$ to generate a random sequence $x^0, x^1, x^2, \dots$, and {FPI} by $\Bar{\opT}$ to generate a sequence $z^0, z^1, z^2, \dots$ with $z^0 = x^0$.
%     Then there exists an asymptotic behavior that happens almost surely : 
%     \[
% \lim_{k\to\infty} \frac{x^k}{k} = -\alpha \bv.
%     \]
Under the conditions of Theorem~\ref{thm:L2} with $\theta \in \left( 0,1\right)$, ${x^k}/{k}$ is strongly convergent to $-\alpha \bv$ in probability $1$. In other words,
    \[
 \frac{x^k}{k} \stackrel{\mathrm{a.s.}}{\rightarrow} -\alpha \bv
    \]
    as $k\to\infty$.
\end{theorem}

While the full proof is presented in the next subsection, here is the outline of the proof of the theorem. Let $z^0,z^1,z^2,\dots$ be the iterates of \eqref{FPI_T}. Assume $\beta < \alpha / \theta$.
From \cref{lem:L2_bound} and further analysis in \cref{appendix:as},
% the last two terms in the right-hand side has an upper bound as a function of $\opS z$, independent from $\opS x$.
% As a result, the following inequality is obtained.
we obtain
\begin{equation*}
    \begin{split}
        &\expec \left[ \Mnorm{ \frac{ x^k} {k} - \frac{ z^k} {k} }^2 \middle| \cF_ {k-1} \right] 
        \leq \Mnorm{ \frac{x^ {k-1}} {k-1} - \frac{z^ {k-1}} {k-1} }^2 + \frac{B}{k^2} \Mnorm{\opS z^0}^2
    \end{split}
\end{equation*}
% \begin{equation*}
%     \begin{split}
%         \expec \left[ \Mnorm{ \frac{ x^k} {k} - \frac{ z^k} {k} }^2 \mid \cF_ {k-1} \right]
%         & \leq 
%         \Mnorm{ \frac{x^ {k-1}} {k-1} - \frac{z^ {k-1}} {k-1} }^2 + \frac{B}{k^2} \Mnorm{\opS z^0}^2
%     \end{split}
% \end{equation*}
for $k=2,3,\dots$,
where $B={\alpha \theta^2 (1 - \alpha\theta) (\beta - \alpha^2)}/{(\alpha - \beta \theta)} \geq 0$
% is a certain constant $\frac{\alpha \theta^2 (1 - \alpha\theta) (\beta - \alpha^2)}{\alpha - \beta \theta}$,
% $x^0, x^1, x^2, \dots$ is a random sequence generated by {RC-FPI} with $\opT$,
% and $z^0, z^1, z^2, \dots$ is a sequence generated by {FPI} with $\Bar{\opT}$ and starting point $z^0 = x^0$, and $\cF_ {k}$ is a filtration consisting of information up to $k$th iteration.
and $\cF_k$ is a filtration consisting of information up to the $k$th iterate.

% Since $\sum_{k=1}^{\infty} \frac{1}{k^2}<\infty$, 
We then apply the Robbins--Siegmund quasi-martingale theorem \citep{robbins1971convergence},
restated as \cref{lem:quasimartigale}, 
to conclude that the random sequence $\Mnorm{ \frac{ x^k} {k} - \frac{ z^k} {k} }^2 $ converges almost surely to some random variable.
Then, by Fatou's lemma and the $L^2$ convergence of \cref{thm:L2}, we have
\[
\expec \left[ \lim_{k\to\infty} \Mnorm{ \frac{ x^k} {k} - \frac{ z^k} {k} }^2 \right] \leq 
\lim_{k\to\infty} \expec \left[  \Mnorm{ \frac{ x^k} {k} - \frac{ z^k} {k} }^2 \right] = 0.
\]

Thus, 
as $k\rightarrow\infty$,
\[
 \Mnorm{ \frac{ x^k} {k} - \frac{ z^k} {k} }^2  \stackrel{\mathrm{a.s.}}{\rightarrow} 0,
\]
and appealing to \cref{thm:pazy}, we conclude the almost sure convergence. Finally, in the case of $\| \cdot\|$-norm, the assumption $\beta<\alpha/\theta$ is satisfied by \cref{lem:expec_norm}.
We defer the detailed proof to \cref{appendix:as}.

% happens, so by combining this with \cref{thm:pazy} gives the almost sure convergence of $\frac{x^k}{k}$ to $- \alpha \bv$ under the condition given in \cref{lem:L2_bound} with $\beta < \alpha / \theta$. Note that such condition is true in Euclidean norm with $\theta<1$, concluding the proof. The detailed proof is provided in \cref{appendix:as}.

\subsection{Proof of \cref{thm:as}} \label{appendix:as}

First, let's recall the Robbins-Siegmund quasi-martingale theorem \citep{robbins1971convergence}.
\begin{lemma} [\citet{robbins1971convergence}] \label{lem:quasimartigale}
    $\cF_1 \subset \cF_2 \subset \dots$ is a sequence of sub-$\sigma$-algebras of $\cF$ where   $\left( \Omega, \cF, P\right)$ is a probability space. When $X_k, b_k, \tau _k, \zeta_k$ are non-negative $\cF_k$-random variables such that
    \[
    \expec \left[ X_{k+1}\mid \cF_k\right] \leq \left( 1+ b_k \right) X_k + \tau _k - \zeta_k,
    \]
    $\lim_{k\to \infty} X_k$ exists and is finite and $\sum_{k=1}^\infty \zeta_k < \infty$ almost surely if  
     $\sum_{k=1}^\infty b_k < \infty, \sum_{k=1}^\infty \tau _k < \infty$.
\end{lemma}

Now, we present a proof of \cref{thm:as} with the norm of $\Mnorm{\cdot}$.

\begin{proposition} \label{lem:as}
    Let $\opT\colon\cH\to\cH$ be $\theta$-averaged with respect to $\Mnorm{\cdot}$ with $\theta\in(0,1]$, and let $x^0, x^1, x^2, \dots$ be the iterates of \eqref{RC-FPI} and let $z^0, z^1, z^2, \dots$ be the iterates of \eqref{FPI_T}. Assume that the distribution of $\ve$ satisfies the uniform expected step-size condition \eqref{condition:P} with $\alpha\in (0,1]$ and \eqref{eq:beta-def} holds with some $\beta$ such that $\beta<\alpha/\theta$.
    Then ${x^k}/{k}$ is strongly convergent to $-\alpha \bv$ in probability $1$, $i.e.$
    \[
    \frac{ x^k} {k} \stackrel{a.s.}{\rightarrow} -\alpha \bv
    \]
    as $k\to\infty$, where $\bv$ is the infimal displacement vector of $\opT$.
\end{proposition}

\begin{proof}
    To use the Robbins-Siegmund quasi-martingale theorem \cref{lem:quasimartigale}, we cannot take full expectation to bound the extra terms in \cref{lem:L2_bound}.
    Here, we provide alternate way to bound the last two terms in \cref{lem:L2_bound}. 

    \begin{equation*}
        \begin{split}
            &-\alpha\theta \left( 1-\alpha \theta \right) \Mnorm{\opS x - \opS z}^2 
            + \theta^2 \left( \beta - \alpha^2\right) \Mnorm{\opS x}^2 \\
            &\hspace{0.5cm} = -  \theta \left(\alpha  -  \beta \theta \right)\Mnorm{\opS x}^2
            +2\alpha\theta \left( 1-\alpha \theta \right) \left< \opS x , \opS z \right>_M
            -\alpha\theta \left( 1-\alpha \theta \right) \Mnorm{ \opS z}^2 \\
            &\hspace{0.5cm} = - \theta \left(\alpha -  \beta \theta \right)\Mnorm{\opS x - \frac{\alpha-\alpha^2 \theta}{\alpha -  \beta\theta}\opS z}^2
            +  \theta \left(\alpha -  \beta \theta \right)\Mnorm{\frac{\alpha-\alpha^2 \theta}{\alpha -  \beta\theta} \opS z}^2
            -\alpha\theta \left( 1-\alpha \theta \right) \Mnorm{ \opS z}^2 \\ 
            &\hspace{0.5cm} = - \theta \left(\alpha -  \beta \theta \right)\Mnorm{\opS x - \frac{\alpha-\alpha^2 \theta}{\alpha -  \beta\theta}\opS z}^2
            + \underbrace{\frac{\alpha \theta^2 (1 - \alpha\theta) (\beta - \alpha^2)}{\alpha - \beta \theta}}_{=: B \geq 0} \Mnorm{ \opS z}^2 \\ 
            &\hspace{0.5cm} \leq B \Mnorm{ \opS z}^2 .
        \end{split}
    \end{equation*}
Note that this inequality only holds when $\beta < \alpha / \theta$. Also, $\beta - \alpha^2\geq 0$ comes from $\Mnorm{\expec\left[ u_\ve \right]}^2\leq \expec\left[ \Mnorm{u_\ve}^2 \right]$.

From \cref{lem:L2_bound} and an inequality $\frac{1}{k} < \frac{1}{k-1}$,
\begin{equation*}
    \begin{split}
        &\expec_{\ve^k  } \left[ \Mnorm{ \frac{ x^k} {k} - \frac{ z^k} {k} }^2 \mid \cF_ {k-1} \right] 
        \leq \Mnorm{ \frac{x^ {k-1}} {k-1} - \frac{z^ {k-1}} {k-1} }^2 + \frac{B}{k^2} \Mnorm{\opS z^ {k-1}}^2,
    \end{split}
\end{equation*}
where $x^0, x^1, x^2, \dots$ is a random sequence generated by \eqref{RC-FPI} with $\opT$, $z^0, z^1, z^2, \dots$ is a sequence generated by \eqref{FPI_T} and starting point $z^0 = x^0$, and $\cF_ {k}$ is a filtration consisting of information up to $n$th iteration.

    With $\Mnorm{\opS z^ {k-1}} \leq \Mnorm{\opS z^0}$ from the \cref{lem:bound_FPI},
    \begin{equation*}
    \begin{split}
        &\expec_{\ve^k  } \left[ \Mnorm{ \frac{ x^k} {k} - \frac{ z^k} {k} }^2 \mid \cF_ {k-1} \right] 
        \leq \Mnorm{ \frac{x^ {k-1}} {k-1} - \frac{z^ {k-1}} {k-1} }^2 + 
        \frac{B}{k^2} \Mnorm{\opS z^0}^2 .
    \end{split}
\end{equation*}

Now we may apply the Robbins-Siegmund quasi-martingale theorem, \cref{lem:quasimartigale}, and conclude that the random sequence $\Mnorm{ \frac{ x^k} {k} - \frac{ z^k} {k} }^2 $ converges almost surely to some random variable since $\sum_{n=1}^{\infty} n^{-2}B \Mnorm{\opS z^0}^2 < \infty$. Then,
\[
\expec \left[ \lim_{k\to\infty} \Mnorm{ \frac{ x^k} {k} - \frac{ z^k} {k} }^2 \right] \leq 
\lim_{k\to\infty} \expec \left[  \Mnorm{ \frac{ x^k} {k} - \frac{ z^k} {k} }^2 \right] = 0
\]
holds, where the inequality comes from the Fatou's lemma and the equality comes from $L^2$ convergence by taking $k\to \infty$ in \cref{lem:L2}. Thus, 
\[
\lim_{k\to\infty} \Mnorm{ \frac{ x^k} {k} - \frac{ z^k} {k} }^2  =0, \qquad a.s.
\]
happens, which, with the strong convergent ${ z^k}/ {k}$ to $-\alpha \bv$ by \cref{thm:pazy}, gives the almost sure convergence of \cref{lem:as}.
\end{proof}

\begin{proof} [Proof of \cref{thm:as}]
 In the case of $M=\opI$, with $\theta \in \left( 0,1\right)$, we have $\beta = \alpha < \alpha / \theta$ from \cref{lem:expec_norm}. Thus, from \cref{lem:as}, we can conclude
    \[
    \frac{ x^k} {k} \stackrel{a.s.}{\rightarrow} -\alpha \bv
    \]
    as $k\to\infty$.
\end{proof}

\subsection{Infeasibility detection.}
% Since ${x^k}/{k} \rightarrow - \alpha \bv$ and since $\bv\ne 0$ implies the problem is inconsistent, we propose
Since $\bv \ne 0$ implies the problem is inconsistent and ${x^k}/{k} \rightarrow - \alpha \bv$, we propose
\begin{equation}
\frac{1}{k}\|x^k\|>\varepsilon 
\label{eq:infeas-detection}
\end{equation}
as a test of inconsistency with sufficiently large $k\in \mathbb{N}$ and sufficiently small $\varepsilon>0$.
The remaining question of how to choose the iteration count $k$ and threshold $\varepsilon$ will be considered in \cref{sec:infeas_detect}.
(This test is not able to detect inconsistency in the pathological case where the problem is inconsistent despite $\bv=0$.)

% \section{Variance of {RC-FPI}} 
\section{Bias and variance of normalized iterates}
\label{sec:variance}

Previously in \cref{sec:linearrate}, we showed that the normalized iterate $x^k/k$ of \eqref{RC-FPI} converges to the scaled infimal displacement vector $-\alpha\bv$. However, to use $x^k/k$ as an estimator of $-\alpha \bv$ and to use $x^k/k\ne 0 $ as a test for inconsistency, we need to characterize the error $\|x^k/k+\alpha \bv\|^2$. In this section, we provide an asymptotic upper bound of the bias and variance of $x^k/k$ as an estimator of $-\alpha\bv$.

% Since $\bv\ne 0$ implies that the problem is inconsistent, we can detect infeasibi
% lity by tracking the normalized iterate of \eqref{RC-FPI}.

% \begin{theorem}
% \label{thm:tv}
% Let $\opT\colon\cH\to\cH$ be $\theta$-averaged with $\theta\in \left( 0,1 \right)$ and the infimal displacement vector $\bv$.
% Assume $\ve^0, \ve^1, \ldots$ is sampled IID from a distribution satisfying the uniform expected step-size condition \eqref{condition:P} with $\alpha\in (0,1]$.
% Let $x^0,x^1,x^2,\ldots$ be the iterates of \eqref{RC-FPI}.
% Then 
%     \[
% \mvar \left(\frac{ x^k} {k} \right) 
% 	\lesssim\frac{(\beta-\alpha^2) \Mnorm{\bv}^2}{k}
%     \]
%     as $k\to\infty$.
% \end{theorem}

\begin{theorem} \label{thm:tv}
    Let $\opT\colon\cH\to\cH$ be $\theta$-averaged with respect to $\Mnorm{\cdot}$ with $\theta\in (0,1]$.
    Let $\bv$ be the infimal displacement vector of $\opT$.
    Assume $\ve^0, \ve^1, \ldots$ is sampled IID from a distribution satisfying the uniform expected step-size condition \eqref{condition:P} with $\alpha\in (0,1]$, and assume \eqref{eq:beta-def} holds with some $\beta>0$ such that $\beta<\alpha/\theta$.
    Let $x^0,x^1,x^2,\ldots$ be the iterates of \eqref{RC-FPI}.
    \begin{itemize}
        \item[(a)]
        If $\bv\in\range(\opI-\opT)$, then as $k\to\infty$,
    \[
        \expec \left[ \Mnorm{ \frac{x^k}{k} + \alpha \bv }^2 \right] \lesssim \frac{(\beta-\alpha^2) \Mnorm{\bv}^2}{k}.
    \]
    \item[(b)]
    In general, regardless of whether $\bv$ is in $\range(\opI-\opT)$ or not,
    \[
\mvar \left(\frac{ x^k} {k} \right) 
	\lesssim\frac{(\beta-\alpha^2) \Mnorm{\bv}^2}{k}
    \]
    as $k\to\infty$.
    \end{itemize}
    % Then
    % More precisely,
    % \[
    %     \limsup_{k\to\infty} \,k\, \expec \left[ \Mnorm{ \frac{x^k}{k} + \alpha \bv }^2 \right] \le (\beta-\alpha^2) \Mnorm{\bv}^2.
    % \]
\end{theorem}

To clarify, the precise meaning of the first asymptotic statement of (a) is
\[
\limsup_{k\to\infty} k \,\expec \left[ \Mnorm{ \frac{x^k}{k} + \alpha \bv }^2 \right]\leq  
 (\beta-\alpha^2) \Mnorm{\bv}^2.
\]
The precise meaning of the asymptotic statement of (b) is defined similarly. 
{
% \color{blue}
When the coordinates represent an orthonormal basis, we have $\Mnorm{\cdot}=\norm{\cdot}$. In this case, we can apply \cref{lem:expec_norm} to choose $\beta$ as $\alpha$. Then, with the uniform expected step-size condition \eqref{condition:P} and the assumption of $\theta \in (0,1)$, same asymptotic bounds of bias and variance can be obtained as ${(\alpha-\alpha^2) \norm{\bv}^2}/{k}$ with by plugging in $\beta\leftarrow \alpha$.
}

Here, we outline the proof for the special case, $\bv \in \range\left( \opI - \opT\right)$, while deferring the full proof 
% without such restriction
to \cref{appendix:tv}. Let $z^0, z^1, z^2, \dots$ be the iterates of \eqref{FPI_T} with $z^0$ satisfying $\theta \opS z^0 = \bv$. Then, $\theta \opS z^k = \bv$ for all $k\in \mathbb{N}$.
Apply \cref{lem:L2_bound} on $x^k$ and $z^k$ and take full expectation to get
\begin{equation*}
    \begin{split}
        \expec \left[ k\Mnorm{\frac{ x^k} {k} - \frac{ z^k} {k} }^2\right]
        \leq
        \frac{1}{k} \Mnorm{x^0-z^0}^2 + \expec \left[ \frac{1} {k} \sum_{j=0}^ {k-1}U^j \right]
    \end{split}
\end{equation*}
where $U^0, U^1, U^2, \dots$ is a sequence of random variables : 
\begin{equation*}
    \begin{split}
        U^k = &-\alpha  \left( \theta^{-1} -\alpha \right) \Mnorm{\theta\opS x^k - \bv}^2+ \theta^2 \left( \beta - \alpha^2\right) \Mnorm{\opS x^k}^2.
    \end{split}
\end{equation*}

The key idea is to bound the $U^k$ terms. This can be done by showing that $\bv$ and $\theta\opS x^k - \bv$ are almost orthogonal. Such property is exhibited in the next two lemmas. 
{
% \color{blue}
Proofs of Lemmas~\ref{lem:e.1} and \ref{lem:lim_orthogonal} are presented in the \ref{appendix:variance.1}.
}

\begin{lemma} \label{lem:e.1}
   Suppose {$\theta$-averaged} operator $\opT\colon\cH\to \cH$ has an infimal displacement vector $\bv$.
   Consider a closed cone $C_\delta$ in $\cH$ with $\delta \in \left( 0,  {\pi}/{2} \right)$,
   % as a set of vectors that forms an angle between $\bv$ to be less then ${\pi}/{2}-\delta$,
   which is a set of vectors whose angle between them and $\bv$ being less than ${\pi}/{2}-\delta$.
\[
C_\delta = \left\{ x: \left< \bv, x \right>_M\geq \sin\delta \Mnorm{\bv} \Mnorm{x} \right\}.
\]
When the points $y,z \in \cH$ satisfy that $\opS y \in \opS z + C_\delta$ and $\opS y \neq \opS z$, then the following inequality holds.
\[
\left< -\bv, y-z \right>_M\leq \cos \delta \Mnorm{\bv} \Mnorm{y-z}.
\]
\end{lemma}

%Next, we propose a lemma to upper bound $\Mnorm{\opS x^k}^2$.
\begin{lemma}
\label{lem:lim_orthogonal}
    Let $\opT\colon\cH\to\cH$ be a {$\theta$-averaged} operator with respect to $\Mnorm{\cdot}$. Let $\bv$ be the infimal displacement vector of $\opT$. Let $\opS = {\theta^{-1}}\left(\opI - \opT \right)$. 
    Consider a sequence $y^0,y^1,y^2,\dots$ in $\cH$ such that its normalized iterate converges strongly to $-\gamma \bv$,
    \[
\lim_{k\to \infty} \frac{y^k} {k} = -\gamma \bv,
    \]
    for some $\gamma>0$.
    Then, for any $\delta\in(0,\pi / 2)$ and $z\in\cH$, there exists $N_{\delta,z} \in \mathbb{N}$ such that, for all $k>N_{\delta, z}$,
    \[
\left< \bv, \opS y^k - \opS z \right>_M  \leq \Mnorm{\bv} \Mnorm{\opS y^k - \opS z} \sin \delta.
    \]
    %for all $k>N_{\delta, z}$.
\end{lemma}

Returning to the proof outline of \cref{thm:tv}, by the Inequality \ref{eq:inf-char} and \cref{lem:lim_orthogonal},
\[
0\leq\left< \bv, \theta\opS x^k - \bv \right>_M \leq  \Mnorm{\bv} \Mnorm{\theta \opS x^k - \bv} \sin \delta
\approx 0
\]
for small $\delta$.
Therefore, for $k$ large enough,
\begin{equation*}
    \begin{split}
\Mnorm{\theta \opS x^k}^2 \lesssim & 
\Mnorm{\bv}^2 +\Mnorm{\theta \opS x^k - \bv}^2 .
% + \bigO(\delta)\Mnorm{\theta \opS x^k - \bv}.
% \\
% & + \left[ \bigO\left(\Mnorm{\theta \opS z^0 - \bv}\right) + \bigO(\delta) \right] \Mnorm{\opS x^k - \opS z^0}
    \end{split}
\end{equation*}
Since $\theta \opS x^k \rightarrow  \bv$ as $k\rightarrow\infty$, we have
% \[
% \left( \beta - \alpha \theta^{-1} \right)\Mnorm{\theta \opS x^k - \bv}^2 + \bigO(\delta)\Mnorm{\theta \opS x^k - \bv},
% \]
% is bounded above uniformly in $k$, we have
\begin{equation*}
    \begin{split}
        &U^k 
        \lesssim \left( \beta - \alpha^2\right) \Mnorm{\bv}^2.
        % +\bigO(\delta).
% +2\theta \tilde{\tau}_{\delta, z^0} \Mnorm{\opS x^k - \opS z^0},
    \end{split}
\end{equation*}
 % where $\Tilde{\tau}_{\delta, z^0}\to 0$ as $\delta\to 0, \theta \opS z^0 \to \bv$.
% Finally, apply Fatou's lemma to conclude that 
% Thus, due to Ces\`{a}ro mean and Fatou's lemma,
Finally, we conclude
 \begin{equation*}
    \begin{split}
&\limsup_{k\to\infty}\, \expec \left[ k\Mnorm{\frac{ x^k} {k} - \frac{ z^k} {k} }^2\right] 
\lesssim  \left( \beta - \alpha^2\right) \Mnorm{\bv}^2.
% + \bigO(\delta).
    \end{split}
\end{equation*}

Note that for (b), $k \mvar \left( {x^k}/{k}\right)$ can be bounded by \[
k \mvar \left( \frac{x^k}{k}\right) \leq \expec\left[ k\Mnorm{\frac{ x^k} {k} - \frac{ z^k} {k} }^2\right],
\]
and so is (a) with an extra $\bigO \left( 1/{\sqrt{k}}\right)$ term.
% Finally, take $\delta\to 0$ to reach the conclusion.
% To prove (a) of \cref{thm:tv}, use
% \[
% k \mvar \left( \frac{x^k}{k}\right) \leq \expec \left[ k\Mnorm{\frac{ x^k} {k} - \frac{ z^k} {k} }^2\right] 
% \]
% and to prove (b) of \cref{thm:tv}, use
% \begin{equation*}
%     \begin{split}
% &\expec \left[ k\Mnorm{ \frac{x^k}{k} + \alpha \bv }^2 \right] 
% \leq 
% \expec\left[ k\Mnorm{\frac{ x^k} {k} - \frac{ z^k} {k} }^2\right] + \bigO \left( \frac{1}{\sqrt{k}}\right),
%     \end{split}
% \end{equation*}
% and take a limit of $\delta\to 0$ and $\Mnorm{\theta \opS z^0 - \bv} \to 0$ to reach the conclusion. The detailed proof is presented in \cref{appendix:tv}.

\subsection{Proof of \cref{thm:tv}} \label{appendix:tv}
\begin{proof}
When $x^0, x^1, x^2, \dots$ is a random sequence generated by {RC-FPI} of $\opT$ and $z^0, z^1, z^2, \dots$ is a sequence generated by {FPI} of $\bar{\opT}$ with $z^0 = z$, from \cref{lem:L2_bound}, it is already known that for all $k$,
\begin{equation*}
\begin{split}
&\expec_{\ve^{k}} \left[ \Mnorm{ \opT_{\ve^ {k}} x^ {k} - \bar{\opT}  z^k } ^2\right] 
\leq  \Mnorm{ x^k- z^k}^2 - \alpha \theta \left( 1 - \alpha \theta \right)  \Mnorm{ \opS  x^k -  \opS  z^k }^2 
+\theta^2 \left( \beta - \alpha^2 \right) \Mnorm{\opS  x^k}^2 .\\
\end{split}
\end{equation*}
For convenience, define 
\[
U^k =  - \alpha \theta \left(1- \alpha \theta \right)  \Mnorm{ \opS x^k -  \opS z^k }^2+\theta^2 \left( \beta - \alpha^2 \right) \Mnorm{\opS x^k}^2.
\]
Consequently, by taking a full expectation,
\begin{equation*}
\begin{split}
&\expec \left[ k \Mnorm{ \frac{x^ {k}}{ k} - \frac{  z^k }{ k} } ^2\right] 
\leq  \frac{1} {k}\Mnorm{x^0-z}^2 
+\expec \left[  \frac{1} {k}\sum_{j=0}^ {k-1} U^j \right] .\\
\end{split}
\end{equation*}

The key of this proof is to bound the term $\Mnorm{\opS  x^k}^2$ in $U^k$ using \cref{lem:lim_orthogonal}, since ${ x^k}/ {k}$ is strongly convergent, i.e., $\lim_{k\to \infty} { x^k} /{k} = -\alpha \bv$ almost surely.
Suppose the case where $\lim_{k\to \infty} { x^k}/ {k} = -\alpha \bv$ holds, which actually does hold almost surely. For such $ x^k$, for an arbitrary $\delta \in \left(0, {\pi}/{2}\right)$, there exists a $N_{\delta,z}$ such that for all $k>N_{\delta, z}$,
\[
\left< \bv, \opS x^k - \opS z \right>_M\leq \sin\delta \Mnorm{\bv} \Mnorm{\opS x^k - \opS z}.
\]
From the inequality above, for all sufficiently large $k>N_{\delta, z}$,
\begin{equation*}
\begin{split}
\Mnorm{\opS x^k}^2 
&= \Mnorm{\opS x^k - \opS z}^2 +2\left<\opS x^k - \opS z, \opS z \right>_M +\Mnorm{\opS z}^2 \\
&= \Mnorm{\opS x^k - \opS z}^2 +2\left<\opS x^k - \opS z, \opS z - \frac{1}{\theta}\bv \right>_M  + 2\left<\opS x^k - \opS z,  \frac{1}{\theta}\bv \right>_M+\Mnorm{\opS z}^2 \\
% &= \Mnorm{\opS x^k - \opS z}^2 \\
% &\hspace{0.5cm} +  2\left< \opS x^k - \opS z, \opS z -\frac{1}{\theta}\bv \right>_M
% +2\left< \opS x^k - \opS z, \frac{1}{\theta}\bv \right>_M \\
% &\hspace{0.5cm} 
% +  \Mnorm{\opS z -\frac{1}{\theta}\bv}^2  +2\left< \frac{1}{\theta}\bv, \opS z -\frac{1}{\theta}\bv \right>_M+ \Mnorm{\frac{1}{\theta}\bv}^2 \\
&\leq \Mnorm{\opS x^k - \opS z}^2  +2 \left\{ \Mnorm{ \opS z -\frac{1}{\theta}\bv} + \sin \delta \Mnorm{\frac{1}{\theta}\bv} \right\}\Mnorm{\opS x^k - \opS z} +\Mnorm{\opS z}^2 .\\
\end{split}
\end{equation*}
 By the inequality above, the term $U^k$ can be bounded as :
\begin{equation*}
\begin{split}
U^k & \leq -\left(\alpha \theta - \alpha^2 \theta^2 \right)\Mnorm{ \opS x^k -\opS z^k }^2 \\
& \hspace{0.5cm} + \theta^2\left( \beta - \alpha^2 \right) \left[ 
\Mnorm{\opS x^k - \opS z}^2  +2 \left\{ \Mnorm{ \opS z -\frac{1}{\theta}\bv} + \sin \delta \Mnorm{\frac{1}{\theta}\bv} \right\}\Mnorm{\opS x^k - \opS z} +\Mnorm{\opS z}^2 \right]\\
&  = \theta^2\left( \beta - \alpha^2 \right) \Mnorm{\opS z}^2
-\theta \left( \alpha -  \beta \theta \right)  \Mnorm{ \opS x^k -  \opS z }^2 \\
&\hspace{0.5cm}+2\theta^2 \left( \beta - \alpha^2 \right) \left\{ \Mnorm{ \opS z -\frac{1}{\theta}\bv} + \sin \delta \Mnorm{\frac{1}{\theta}\bv} \right\}\Mnorm{\opS x^k - \opS z}\\
&\hspace{0.5cm} +\left(\alpha \theta - \alpha^2 \theta^2 \right) \left\{   \Mnorm{ \opS x^k -  \opS z }^2  - \Mnorm{ \opS x^k -  \opS z^k }^2 \right\}. \\
\end{split}
\end{equation*}
Here, the term $\Mnorm{ \opS x^k -  \opS z }^2  - \Mnorm{ \opS x^k -  \opS z^k}^2$ is bounded above by
\begin{equation*}
    \begin{split}
        &\Mnorm{ \opS x^k -  \opS z }^2 -\Mnorm{ \opS x^k -   \opS z^k }^2 \\
&\hspace{0.5cm}\le \Mnorm{ \opS x^k -  \opS z }^2 -\left( \Mnorm{ \opS x^k -   \opS z }^2 + \Mnorm{ \opS z -   \opS z^k }^2 + 2\left<\opS x^k -   \opS z  , \opS z -   \opS z^k   \right>_M\right) \\
& \hspace{0.5cm}\le -\Mnorm{ \opS z-   \opS z^k }^2 + 2\Mnorm{ \opS x^k -  \opS z}\Mnorm{ \opS z-   \opS z^k},
    \end{split}
\end{equation*}
from the triangular inequality. 

Thus, we can bound $U^k$ as 
\begin{equation*}
\begin{split}
U^k &  \leq \theta^2\left( \beta - \alpha^2 \right) \Mnorm{\opS z}^2
-\theta\left( \alpha -  \beta\theta \right)  \Mnorm{ \opS x^k -  \opS z }^2 \\
&\hspace{0.5cm}+2\theta^2 \left( \beta - \alpha^2 \right) \left\{ \Mnorm{ \opS z -\frac{1}{\theta}\bv} + \sin \delta \Mnorm{\frac{1}{\theta}\bv} \right\}\Mnorm{\opS x^k - \opS z}\\
&\hspace{0.5cm} +\left(\alpha \theta - \alpha^2 \theta^2 \right) 
\left\{  
-\Mnorm{ \opS z-\opS z^k }^2 +  2\Mnorm{ \opS x^k -  \opS z}\Mnorm{ \opS z-   \opS z^k}
\right\} \\
& =  \theta^2\left( \beta - \alpha^2 \right) \Mnorm{\opS z}^2 -\left(\alpha \theta - \alpha^2 \theta^2 \right)    \Mnorm{ \opS z-  \opS  z^k }^2 \\
&\hspace{0.5cm}-\theta\left( \alpha -  \beta \theta \right)  \Mnorm{ \opS x^k -  \opS z }^2 
+2\theta \tau_{\delta, z, k} \Mnorm{\opS x^k - \opS z}\\
& \le  \theta^2\left( \beta - \alpha^2 \right) \Mnorm{\opS z}^2 \\
&\hspace{0.5cm}-\theta\left( \alpha -  \beta \theta \right)  \Mnorm{ \opS x^k -  \opS z }^2 
+2\theta \tau_{\delta, z, k} \Mnorm{\opS x^k - \opS z},\\
% & \leq  \theta^2\left( \beta - \alpha^2 \right) \Mnorm{\opS z}^2  
% -\left(\alpha \theta - \alpha^2 \theta^2 \right)    \Mnorm{ \opS z-  \opS  z^k }^2
% +\frac{ \theta}{\alpha - \beta \theta}
%  \tau_{\delta, z, k}^2,
\end{split}
\end{equation*}
where $\tau_{\delta, z, k}$ is defined as
\[
\tau_{\delta, z, k} =  \left( \beta - \alpha^2 \right) \left( \Mnorm{\theta \opS z - \bv} + \sin \delta \Mnorm{\bv} \right) 
+ \left(\alpha-\alpha^2 \theta \right)  \Mnorm{ \opS z-  \opS  z^k }.
\]

To make an upper bound of $\tau_{\delta, z, k}$ regardless of $k$, an upper bound of $\Mnorm{ \opS z-  \opS z^k }$ independent from $k$ is required. Since $\bv$ is the infimal displacement vector, 
\[
\left< \opS z^k-\frac{1}{\theta}\bv, \bv  \right>_M\geq 0, 
\quad \Mnorm{  \opS  z }^2 -\Mnorm{  \frac{1}{\theta}\bv}^2 \geq 0,
\]
hold. Recall that $z^0=z$. With $\Mnorm{\opS z^k} \leq \Mnorm{\opS z^0}= \Mnorm{\opS z}$ from \cref{lem:bound_FPI}, such uniform upper bound can be built as 
\begin{equation*}
\begin{split}
\Mnorm{ \opS z-  \opS z^k} 
&\leq  \Mnorm{ \opS z- \frac{1}{\theta}\bv } + \Mnorm{  \opS z^k -\frac{1}{\theta}\bv} \\
&=  \Mnorm{ \opS z- \frac{1}{\theta}\bv } + 
\sqrt{
\Mnorm{\opS  z^k}^2
-2\left< \opS  z^k, \frac{1}{\theta}\bv \right>_M
+\Mnorm{  \frac{1}{\theta}\bv}^2 } \\
&=  \Mnorm{ \opS z- \frac{1}{\theta}\bv } + 
\sqrt{
\Mnorm{\opS  z^k}^2
-2\left< \opS  z^k - \frac{1}{\theta}\bv, \frac{1}{\theta}\bv \right>_M
-\Mnorm{  \frac{1}{\theta}\bv}^2 } \\
&\leq  \Mnorm{ \opS z- \frac{1}{\theta}\bv } + 
\sqrt{\Mnorm{  \opS  z }^2 -\Mnorm{  \frac{1}{\theta}\bv}^2}.
\end{split}
\end{equation*}
Now define $\tilde{\tau}_{\delta, z}$ as
\[
\tilde{\tau}_{\delta, z}
= \left( \beta - \alpha^2 \right) \left( \Mnorm{\theta \opS z - \bv} + \sin \delta \Mnorm{\bv} \right) 
+ \left(\alpha-\alpha^2 \theta \right)  \left\{ 
\Mnorm{ \opS z- \frac{1}{\theta}\bv } + 
\sqrt{\Mnorm{  \opS  z }^2 -\Mnorm{  \frac{1}{\theta}\bv}^2}
\right\},
\]
then we have $\tau_{\delta, z, k} \leq \tilde{\tau}_{\delta, z}$ for any $k\in \mathbb{N}$.

Thus, we can bound  $U^k$ as 
\begin{equation*}
    \begin{split}
    U^k 
        &\leq \theta^2\left( \beta - \alpha^2 \right) \Mnorm{\opS z}^2
        -\theta\left( \alpha -  \beta \theta \right)  \Mnorm{ \opS x^k -  \opS z }^2 
        +2\theta \tau_{\delta, z, k} \Mnorm{\opS x^k - \opS z} \\
        &\leq \theta^2\left( \beta - \alpha^2 \right) \Mnorm{\opS z}^2
        -\theta\left( \alpha -  \beta \theta \right)  \Mnorm{ \opS x^k -  \opS z }^2 
        +2\theta \tilde{\tau}_{\delta, z} \Mnorm{\opS x^k - \opS z}. \\
    \end{split}
\end{equation*}

From the fact that $-at^2+2bt \leq {b^2}/{a}$ for any $a,b>0$ and $t\in\reals$, $U^k$ has an upper bound which is completely independent from $x^k$ and $k$.
\[
U^k \leq \theta^2\left( \beta - \alpha^2 \right) \Mnorm{\opS z}^2  + \frac{ \theta}{\alpha-\beta\theta}
 \tilde{\tau}_{\delta, z}^2
\]

However, this upper bound holds only at $k>N_{\delta, z}$. Since $N_{\delta, z}$ depends on the choice of the sequence $x^0, x^1, x^2, \dots$, such upper bound only works when the sequence $x^0, x^1, x^2, \dots$ is fixed. To avoid this problem, take a limit supremum of $U^k$ over $k$,
\begin{equation*}
\begin{split}
\limsup_{k\to\infty} U^k 
& \leq  \theta^2\left( \beta - \alpha^2 \right) \Mnorm{\opS z}^2  
+ \frac{ \theta}{\alpha-\beta\theta}
 \tilde{\tau}_{\delta, z}^2.
\end{split}
\end{equation*}
Furthermore, due to Ces\`{a}ro mean,
\begin{equation*}
\begin{split}
\limsup_{k\to\infty} \left\{  \frac{1} {k}\Mnorm{x^0-z}^2 + \frac{1} {k}\sum_{j=0}^ {k-1} U^j \right\} 
\leq \theta^2\left( \beta - \alpha^2 \right) \Mnorm{\opS z}^2  
+ \frac{ \theta}{\alpha-\beta\theta}
 \tilde{\tau}_{\delta, z}^2.\\
\end{split}
\end{equation*}

Now, this inequality always holds with only one condition on the choice of sequence $x^0, x^1, x^2 \dots,$
\[
\lim_{k\to \infty} \frac{ x^k} {k} = -\alpha \bv,
\]
 regardless of $z, \delta$. Since $\lim_{k\to \infty} { x^k}/ {k} = -\alpha \bv$ holds almost surely, 
\[
\expec \left[ 
\limsup_{k\to\infty} \left\{  \frac{1} {k}\Mnorm{x^0-z}^2 + \frac{1} {k}\sum_{j=0}^ {k-1} U^j \right\} 
 \right]
 \leq \theta^2\left( \beta - \alpha^2 \right) \Mnorm{\opS z}^2  
+ \frac{ \theta}{\alpha-\beta\theta}
 \tilde{\tau}_{\delta, z}^2
\]
holds almost surely for any $z, \delta$.
 
By Fatou's lemma, we also have
\[
 \limsup_{k\to\infty} \expec \left[ 
   \frac{1} {k}\Mnorm{x^0-z}^2 + \frac{1} {k}\sum_{j=0}^ {k-1} U^j  
 \right] 
 \leq \expec \left[ 
\limsup_{k\to\infty} \left\{  \frac{1} {k}\Mnorm{x^0-z}^2 + \frac{1} {k}\sum_{j=0}^ {k-1} U^j \right\} 
 \right].
\]

Thus, $\limsup_{k\to\infty}\expec \left[ k \Mnorm{ \frac{x^ {k}} {k} - \frac{  z^k } {k} } ^2\right]  $ has a upper bound of
 \begin{equation} \label{eq:limsup_bound}
 \begin{split}
\limsup_{k\to\infty}\expec \left[ k \Mnorm{ \frac{x^ {k}} {k} - \frac{  z^k } {k} } ^2\right] 
&\leq \limsup_{k\to\infty} \expec \left[ 
   \frac{1} {k}\Mnorm{x^0-z}^2 + \frac{1} {k}\sum_{j=0}^ {k-1} U^j  
 \right] \\
&\leq \theta^2\left( \beta - \alpha^2 \right) \Mnorm{\opS z}^2  
+ \frac{ \theta}{\alpha-\beta\theta}
 \tilde{\tau}_{\delta, z}^2.
\end{split}
\end{equation}

\textbf{Proof of statement (b).}
First, let's prove the statement (b) of \cref{thm:tv}. Since 
\[
\limsup_{k\to\infty} k\mvar\left(\frac{ x^k } {k}\right) 
\leq \limsup_{k\to\infty}\expec \left[ k \Mnorm{ \frac{x^ {k}} {k} - \frac{  z^k } {k} } ^2\right],
\]
from \eqref{eq:limsup_bound} we have
\[
\limsup_{k\to\infty} k\mvar\left(\frac{ x^k } {k}\right) 
\leq \theta^2\left( \beta - \alpha^2 \right) \Mnorm{\opS z}^2  
+ \frac{ \theta}{\alpha-\beta\theta}
 \tilde{\tau}_{\delta, z}^2.
\]

At start, we chose $z$ and $\delta$ arbitrarily. Since $\bv$ is the infimal displacement vector, there exists a sequence of $z$'s that allows us to take a limit $\Mnorm{\opS z - {\theta^{-1}}\bv} \to 0$. When we take a limit $\Mnorm{\opS z - {\theta^{-1}}\bv} \to 0$ and $\delta \to 0$,
\[
\lim_{\Mnorm{\opS z - \frac{1}{\theta}\bv} \to 0, \delta \to 0} \tilde{\tau}_{\delta, z} =0.
\]

Since $\limsup_{k\to\infty} k\mvar\left({ x^k}/ {k}\right)$ is independent from $\delta$ and $z$, by $\Mnorm{\opS z - {\theta^{-1}}\bv} \to 0$ and $\delta \to 0$ we have 
\begin{equation*}
\limsup_{k\to\infty} k\mvar\left(\frac{ x^k} {k}\right)\leq\left( \beta - \alpha^2 \right) \Mnorm{\bv}^2.
\end{equation*}

\textbf{Proof of statement (a).}
Next, to prove the statement (a) of \cref{thm:tv}, let's start again from inequality \eqref{eq:limsup_bound},  
 \begin{equation*} 
 \begin{split}
\limsup_{k\to\infty}\expec \left[ k \Mnorm{ \frac{x^ {k}} {k} - \frac{  z^k } {k} } ^2\right] 
&\leq \theta^2\left( \beta - \alpha^2 \right) \Mnorm{\opS z}^2  
+ \frac{ \theta}{\alpha-\beta\theta}
 \tilde{\tau}_{\delta, z}^2.
\end{split}
\end{equation*}
Expand the term $\Mnorm{ \frac{x^k}{k} + \alpha \bv }^2$ as 
\[
\expec \left[ \Mnorm{ \frac{x^k}{k} + \alpha \bv }^2 \right]
=\expec \left[ \Mnorm{ \frac{x^k}{k} -\frac{z^k}{k} }^2 \right]
+2 \left<\expec \left[\frac{x^k}{k} -\frac{z^k}{k}\right], \frac{z^k}{k} +\alpha \bv \right>_M
+\Mnorm{ \frac{z^k}{k} +\alpha \bv}^2 .
\]
From \cref{lem:L2} we have
\begin{equation*}
    \begin{split}
        \Mnorm{\expec \left[ \frac{x^k}{k} -\frac{z^k}{k}\right]}^2
&\leq \expec \left[ \Mnorm{\frac{x^k}{k} -\frac{z^k}{k}}^2\right]  \\
&\leq\frac{1} {k} \left( 1-\alpha \theta \right) \left[ 2\sqrt{\alpha\theta}  \Mnorm{\opS x^0}\Mnorm{\opS z^0}
- \frac{\alpha}{\theta}  \Mnorm{\bv}^2 \right]  +\frac{1}{k^2}\Mnorm{x^0 - z^0}^2.
    \end{split}
\end{equation*}

When $x_\star$ is a point such that $x_\star-\opT x_\star = v$, set $z^0 = x_\star$. Then,
\[
z^k = -k \alpha \bv + x_\star,
\]
since $\Mnorm{\theta \opS z^k} \leq\Mnorm{\theta \opS z} = \Mnorm{\bv}$ makes $\theta \opS z^k = \bv$ for all $k \in \mathbb{N}$. Thus, 
\begin{equation*}
    \begin{split}
        &\expec \left[ \Mnorm{ \frac{x^k}{k} + \alpha \bv }^2 \right]^2
\leq \expec \left[ \Mnorm{ \frac{x^k}{k} -\frac{z^k}{k} }^2 \right] \\
&\hspace{0.5cm}+2 \frac{1}{\sqrt{k}}\sqrt{\left( 1-\alpha \theta \right) \left[ 2\sqrt{\alpha\theta}  \Mnorm{\opS x^0}\Mnorm{\opS x_\star}
- \frac{\alpha}{\theta}  \Mnorm{\bv}^2 \right]  +\frac{1}{k}\Mnorm{x^0 - z^0}^2} \Mnorm{\frac{x_\star}{k}}
+\Mnorm{\frac{x_\star}{k}}^2 .
    \end{split}
\end{equation*}
Note that the last two terms are $\bigO\left(k^{-3/2}\right)$. By taking $\limsup$ as $k\to \infty$,
\[
\limsup_{k\to\infty} \expec \left[ k \Mnorm{ \frac{x^k}{k} + \alpha \bv }^2 \right]
\leq
\limsup_{k\to\infty}\expec \left[ k \Mnorm{ \frac{x^ {k}} {k} - \frac{  z^k } {k} } ^2\right].
\]

Thus,
 \begin{equation*} 
\limsup_{k\to\infty}\expec \left[ k \Mnorm{ \frac{x^ {k}} {k} +\alpha \bv } ^2\right] 
\leq \theta^2\left( \beta - \alpha^2 \right) \Mnorm{\opS x_\star}^2  
+ \frac{ \theta}{\alpha-\beta\theta}
 \tilde{\tau}_{\delta, x_\star}^2.
\end{equation*}
Since $\theta \opS x_\star = \bv$, by $\delta \to 0$, we have $\tilde{\tau}_{\delta, x_\star} \to 0$ and
 \begin{equation*} 
\limsup_{k\to\infty}\expec \left[ k \Mnorm{ \frac{x^ {k}} {k} +\alpha \bv } ^2\right] 
\leq \left( \beta - \alpha^2 \right) \Mnorm{\bv}^2.
\end{equation*}

\end{proof}

\subsection{Tightness of variance bounds} \label{sec:experiments}

In this section, we provide examples for which the variance bound of \cref{thm:tv} holds with equality and with a strict inequality. We then discuss how the geometry of $\range(\opI-\opT)$ influences the tightness of the inequality.
Throughout this section, we consider the setting where the norm and inner product is $\| \cdot\|$-norm and $\langle \cdot, \cdot\rangle$, with $\cH = \RR^m$, $\cH_i = \RR$, and $\ve$ follows uniform distribution on the set of standard unit vectors of $\cH$.
In this case, the smallest $\beta$ we can choose is $\alpha = 1/m$.

\subsubsection{Example: \cref{thm:tv}(b) holds with equality.}
Consider the translation operator $\opT(x)=x-\bv$.
% then the variance of $\frac{x^k}{k}$ at any iteration count $k$ is exactly
When $x^0, x^1, x^2, \dots $ are the iterates of \eqref{RC-FPI} with $\opT$, then
\[
k \mvar \left( \frac{x^k} {k} \right) = \alpha \left(1-\alpha \right) \norm{\bv}^2
\]
for $k=1,2,\dots$, and the variance bound of \cref{thm:tv} holds with equality. 

% \textcolor{blue}
{
Furthermore, for any choice of the distribution of $\mathcal{I}$ satisfying the uniform expected step-size condition \eqref{condition:P}, we can always find a translation operator whose $k \mathrm{Var}_M  \left(\frac{x_k}{k}\right)$ exactly obtaining the upper bound of \cref{thm:tv}.
Fixing the distribution of $\mathcal{I}$, and choosing the smallest possible $\beta$, there exists $u\neq 0$ such that
\[
\expec[\Mnorm{u_\ve}^2] = \beta \Mnorm{u}^2,
\]
and \eqref{RC-FPI} with the translation operator $\opT: x\mapsto x-u$ leads to the variance
\[
k\mvar\left( \frac{x_k}{k}\right) = \mvar(u_\ve) = \left(\beta - \alpha^2\right) \Mnorm{u}^2 = \left(\beta - \alpha^2\right) \Mnorm{\bv}^2.
\]
% Furthermore, consider a different distribution of $\ve$ as $\ve\sim \text{Uniform}(\{0,0.5\}^m)$. Then, 
%         \[
%         \alpha = \frac{1}{4}, \qquad \beta = \frac{1}{8}, \qquad k\var\left( \frac{x_k}{k}\right) = \frac{1}{16}\norm{\bv}^2.
%         \]
%         This is also an example of tight variance bound. But, the different setting on the distribution of $\ve$ leads to the different values of $\alpha$ and $\beta$. Furthermore, while the choice of $\beta = 1/8$ gives the tight variance bound, one can also choose $\beta = \alpha$ using the \cref{lem:expec_norm}. If $\beta = \alpha$ is used, then the variance bound is no longer tight.
}

\subsubsection{Example: \cref{thm:tv}(b) holds with strict inequality.}
Define $\opT\colon\RR^2 \to \RR^2$ as   
\[
    \opT \colon \left( x, y \right) \mapsto \left( x-\frac{1+x-y}{2},\, y-\frac{1+y-x}{2}\right),
\]
which is {${1}/{2}$-averaged} and has the infimal displacement vector $\left( {1}/{2}, {1}/{2}\right)$. 

When $(x^0, y^0), (x^1,y^1), (x^2,y^2), \dots $ are the iterates of \eqref{RC-FPI} with $\opT$, then
\begin{equation} \label{eq:5.1}
    \limsup_{k\to\infty}\, k \mvar \left(\frac{\left( x^k, y^k\right)} {k} \right) = \frac{1}{24}.
\end{equation}

On the other hand, the right hand side of the inequality in \cref{thm:tv} (b) is
\[
\alpha \left(1-\alpha \right) \norm{\bv}^2 =\frac{1}{2}\left(1-\frac{1}{2}\right) \norm{\bv}^2=\frac{1}{8}.
\]
{
% \color{blue}
The detailed computation of the \cref{eq:5.1} is presented in the \ref{appendix:variance.2}.
}

\subsection{Relationship between the variance and the range set.}

Consider the three convex sets $A$, $B$, and $C$ in \cref{fig:convex_sets} as a subset of $\cH = \RR^2$.
The explicit definitions are
\begin{equation*}
    \begin{split}
        & A = \left\{ (x,y) \mid x \leq -10,\, y \leq -5 \right\} \\
        & B = \left\{ (x,y) \mid \dist\left( (x,y), 3A \right) \leq 2\sqrt{5^2 + 10^2} \right\}\\
        & C = \left\{ (x,y) \mid -2x -y \geq 25 \right\},
    \end{split}
\end{equation*}
where $\dist\left( (x,y), 3A \right)$ denotes the (Euclidean) distance of $(x,y)$ to the set $3A=\{(3x,3y)\,|\,(x,y)\in A\}$. The minimum norm elements in each set are all identically equal to $(-10, -5)$.

Let $\opT = \opI-\theta\Proj$, where $\Proj$ denotes the projections onto $A$, $B$, and $C$.
Then $\opT$ is $\theta$-averaged and $\range(\theta^{-1}(\opI - \opT))$ is equal to $A$, $B$, and $C$, respectively.
% The three sets are designed so that $\opT$, in all three cases, have the same infimal displacement vector.
These sets are designed for $\opT$ to have the same infimal displacement vector.
\cref{fig:variance} (left), shows that the normalized iterates of the three instances have different asymptotic variances despite identical $\bv$. 
In the experiment, $\theta$ was set as $0.2$, and as a consequence, $\bv = (-2,-1)$ is the infimal displacement vector for each experiments.
\eqref{RC-FPI} is performed with $x^0 = (0,0)$, $m=2$ and $\cH_1 = \cH_2 = \RR$.

\begin{figure}[htbp]
  \centering
  \vspace{-5pt}
  \subfigure[$A$]{\includegraphics[width=0.25\textwidth]{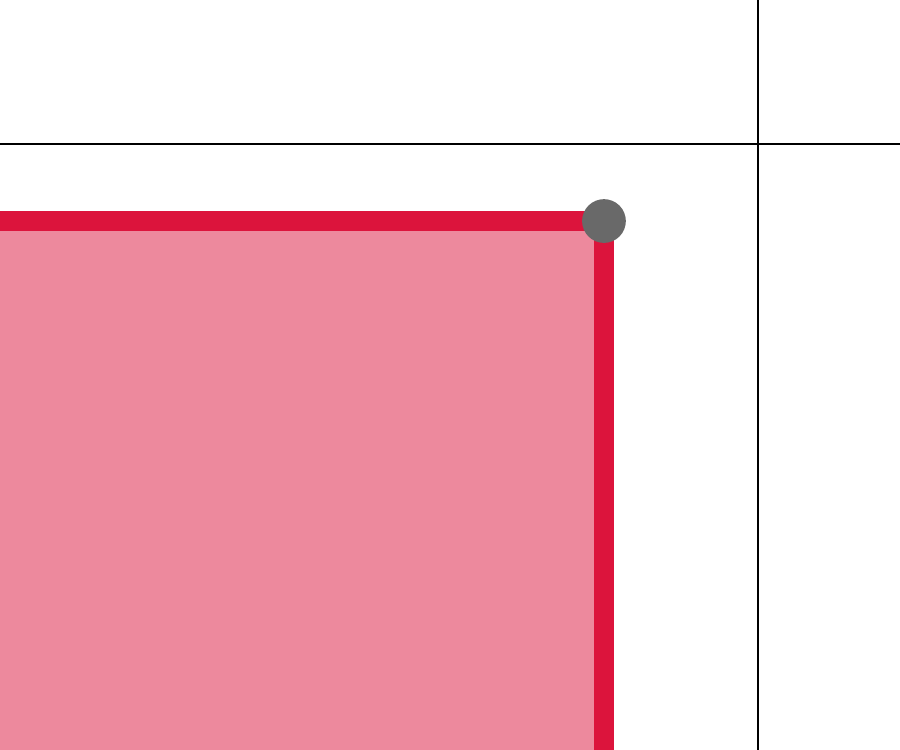}}
  \hspace{.02\textwidth}
  \subfigure[$B$]{\includegraphics[width=0.25\textwidth]{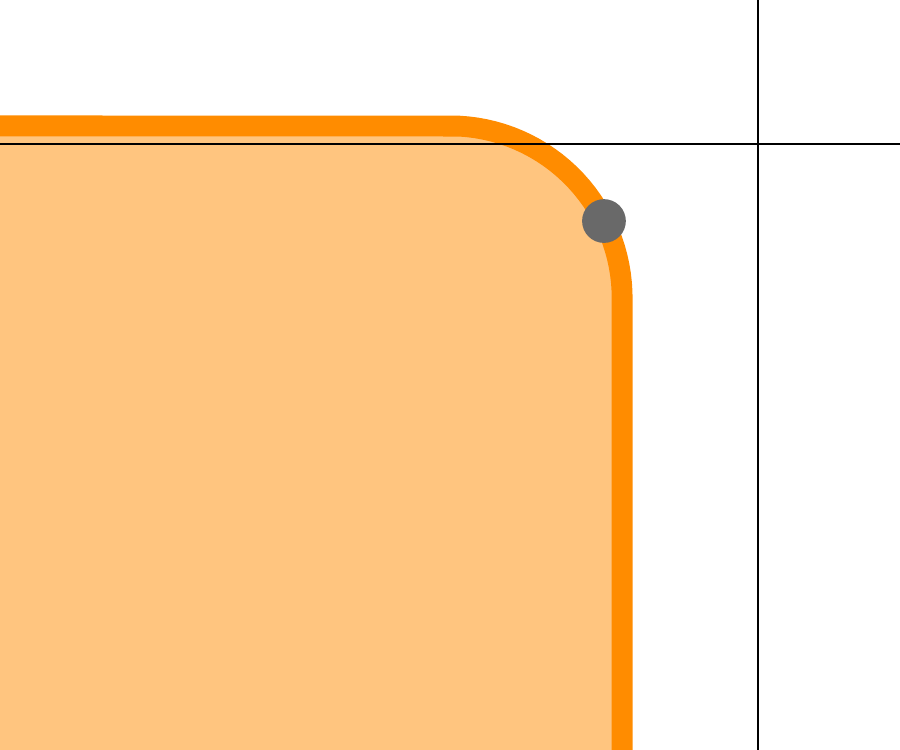}}
  \hspace{.02\textwidth}
  \subfigure[$C$]{\includegraphics[width=0.25\textwidth]{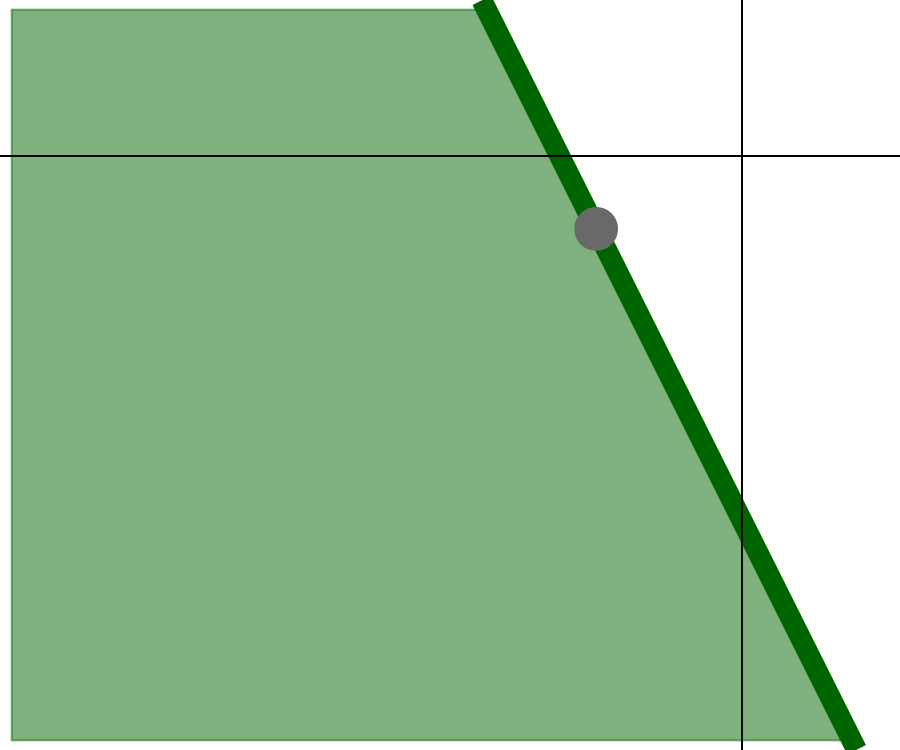}}
  %\subfigure[$D$]{\includegraphics[width=0.11\textwidth]{./Td.pdf}}
  \vspace{-5pt}
  \caption{Visualization  $A$, $B$, and $C$ as defined in  \cref{sec:experiments}. The grey dot is $\theta^{-1}\bv$, where $\bv$ is the infimal displacement vector of $\opT = \opI-\theta\Proj$.
  }
  \label{fig:convex_sets}
\end{figure}

% with each set $A,B,C$ is performed, we get the results of \cref{fig:variance}. We observe that while $A$ obtains near upper bound value in \cref{thm:tv}, $B,C$ obtain significantly smaller variances.

% Consider three $\theta$-averaged operators with identical infimal displacement vector and perform \eqref{RC-FPI}. The three convex sets $A,B,C$ in \cref{fig:convex_sets} are the range set of $\theta^{-1}(\opI - \opT)$. 

\begin{figure*}[htbp]
\centering
\raisebox{-0.5\height}{\includegraphics[width=0.45\textwidth]{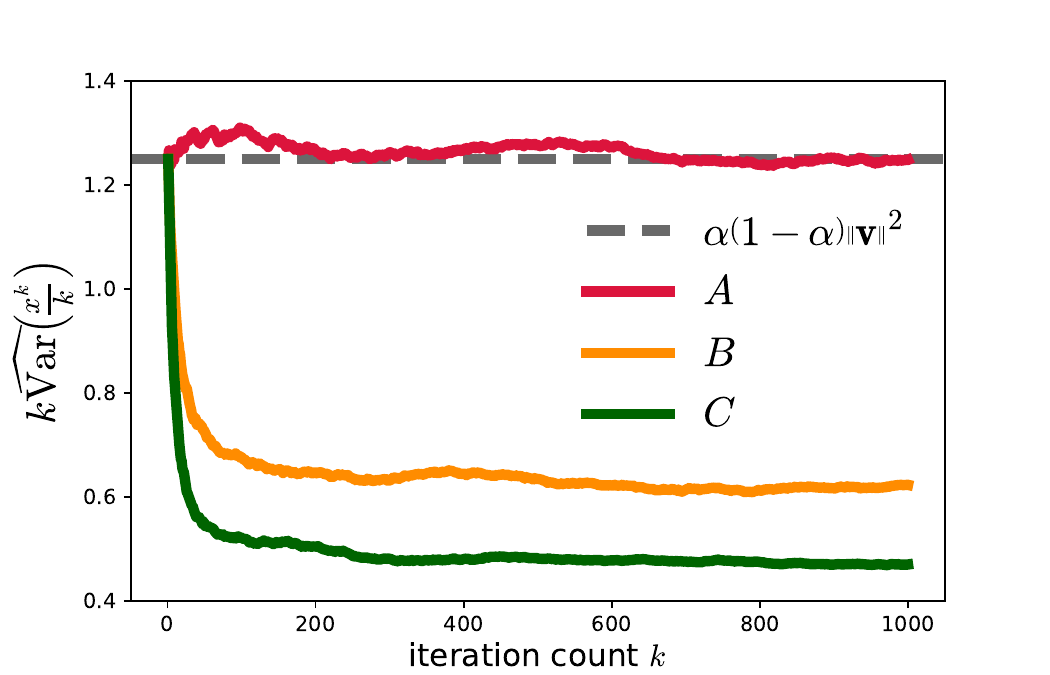}}
  \hspace{.05\textwidth}
 \raisebox{-0.5\height}{\includegraphics[width=0.35\textwidth]{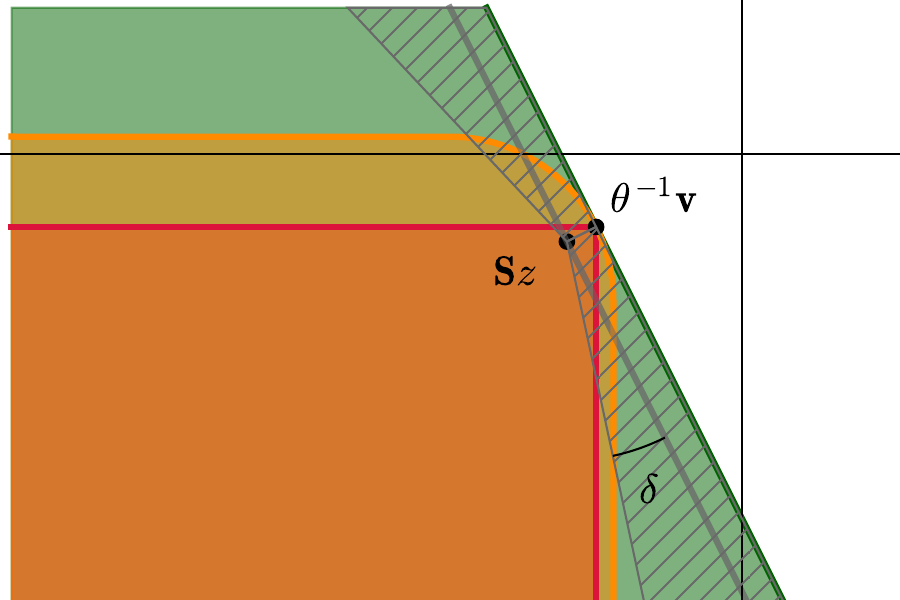}}
  \caption{
      (Left) Graph of $k \widehat{\mathrm{Var}}\left( {x^k}/{k}\right)$ by $k$, where $\widehat{\mathrm{Var}}\left( {x^k}/{k}\right)$ is the variance estimate with 10,000 samples. 
      (Right) Visualization of  $A$, $B$, and $C$ as red, yellow, green regions and $D_{z,\delta}$ as the hatched area, where the sets are as defined in  \cref{sec:experiments}.
      We conjecture that the broader intersection with $D_{z,\delta}$ leads to smaller asymptotic variance. (For interpretation of the colors in the figures, the
reader is referred to the web version of this article.)
      }
  \label{fig:variance}
\end{figure*}

% $B\cap D_{z,\delta}$, and
% $C\cap D_{z,\delta}$.

% in order to have smaller asymptotic variance compared to the upper bound ${\small (\alpha-\alpha^2) \|\bv\|^2}$, the region of the range set overlapping with
% for small $\delta>0$, necessarily have to be large.
% Such area is the candidate area for $\opS x^k$, since $\opS x^k$ exists in the intersection of the range set and the area of \eqref{eq:hyperplane} as a consequence of \cref{lem:lim_orthogonal}.

% \begin{figure}[htbp]
% \centering
% \includegraphics[width=0.4\textwidth]{./ray.pdf}
% \caption{Visualization of set $A, B, C$, along with the hatched area of $\left\{ u\in \RR^2 \middle| \left< \bv, u- \opS z \right>\leq \norm{\bv} \norm{u- \opS z}\sin \delta \right\} \cap C$.
% }
% \label{fig:ray}
% \end{figure}

% In the proof of \cref{thm:tv} presented in \cref{appendix:tv}, we could find some explanation related to this tendency.
%     This tendency aligns well with the proof of \cref{thm:tv} presented in \cref{appendix:tv}. 
% We bounded the term
% \begin{equation*}
%     \begin{split}
% -\theta\left( \alpha -  \beta \theta \right)  \Mnorm{ \opS x^k -  \opS z }^2 
% &+2\theta \tau_{\delta, z, k} \Mnorm{\opS x^k - \opS z}\\
%     \end{split}
% \end{equation*}
% by $\frac{ \theta}{\alpha - \beta \theta}  \tau_{\delta, z, k}^2$ to obtain the upper bound in \cref{thm:tv}.
% Here, $\tau_{\delta, z, k} \to 0$ as $\delta \to 0$ and $\theta \opS z^k \to \bv$.

We conjecture that the asymptotic variance is intimately related to the geometry of the set $\range(\theta^{-1}(\opI - \opT))$.
For $z\in \mathbb{R}^n$ and $\delta>0$, let
\[
D_{z,\delta}=     \left\{ u\in \RR^2 \middle| \left< \bv, u- \opS z \right>\leq \norm{\bv} \norm{u- \opS z}\sin \delta \right\}.
\]
\cref{lem:lim_orthogonal} states that eventually, $\opS x^k\in D_{z,\delta}$ for sufficiently large $k$.
Since $\opS x^k \in \range (\theta^{-1}(\opI - \opT))$ for all $k$, the shaded region in the \cref{fig:variance} (right) depicting $D_{z,\delta}\cap  \range(\theta^{-1}(\opI - \opT))$ actually shows the region where $\opS x^k$ lies for large $k$.
In the proof of  \cref{thm:tv}, loosely speaking, we establish the upper bound using \[
-\theta\left( \alpha -  \beta \theta \right)  \Mnorm{ \opS x^k -  \theta^{-1} \bv }^2 \leq 0.
\]
Therefore, the variance can be strictly smaller than the upper bound when $\Mnorm{ \opS x^k -  \theta^{-1} \bv }^2$ is large, which can happen when the area of intersection $D_{z,\delta}\cap  \range(\theta^{-1}(\opI - \opT))$ is large near $\theta^{-1} \bv$.
This can be observed in \cref{fig:variance}, which shows that the range set having large intersection with $D_{z,\delta}$ have smaller asymptotic variance.

% Thus, if the candidate area for $\opS x^k$ is small near $\theta^{-1} \bv$ just like the case $A$, then inequality \eqref{eq:conject} gets tight, giving the variance close to the upper bound.  On the other hand, to have a significantly strict inequality in \eqref{eq:conject}, to have a smaller variance, the region of intersection of $\range \left( \opI - \opT \right)$ with the hatched region of \cref{fig:variance} must be large.

\section{Infeasibility detection}
\label{sec:infeas_detect}
In this section, we present the infeasibility detection method for \eqref{RC-FPI} using the hypothesis testing.

\begin{theorem}
\label{thm:detection}
Let $\opT\colon\cH\to\cH$ be $\theta$-averaged with respect to $\Mnorm{\cdot}$ with $\theta\in (0,1]$.
    Let $\bv$ be the infimal displacement vector of $\opT$.
    Assume $\ve^0, \ve^1, \ldots$ is sampled IID from a distribution satisfying the uniform expected step-size condition \eqref{condition:P} with $\alpha\in (0,1]$, and assume \eqref{eq:beta-def} holds with some $\beta>0$ such that $\beta<\alpha/\theta$.
    Let $x^0,x^1,x^2,\ldots$ be the iterates of \eqref{RC-FPI}.
    Then under the null hypothesis $\Mnorm{\bv}\leq \delta$ with nonzero $\delta$, for $\epsilon$ that satisfies $\alpha\delta < \epsilon$,
    \[
    \mathbb{P}\left(  \Mnorm{\frac{x^k}{k} } 
\geq \varepsilon \right) \lesssim \frac{\left( \beta - \alpha^2\right)\delta^2}{k(\varepsilon-\alpha\delta)^2}
    \]
    as $k\to\infty$, where $\bv$ is the infimal displacement vector of $\opT$.
\end{theorem}

Therefore, for any statistical significance level $p\in (0,1)$, the test
\[
\Mnorm{\frac{x^k}{k} } 
\geq \varepsilon 
\]
with 
\[
k\gtrsim \frac{\left( \beta - \alpha^2\right) \delta^2}{p\left(\varepsilon - \alpha \delta \right)^2}
\]
can reject the null hypothesis and conclude that ${\Mnorm{ \bv} >\delta}$, which implies that the problem is inconsistent.

{
% \color{blue}
We further discuss this iteration requirement.
Let's rewrite the bound as
\[
    k \gtrsim 
    \frac{(\beta-\alpha^2) \delta^2}{p(\varepsilon-\alpha\delta)^2}
    =
    \frac{(\beta-\alpha^2)}{p(\varepsilon/\delta - \alpha)^2}.
\]
$\alpha$ and $\beta$ are parameters satisfying $\beta\ge\alpha^2$, which comes from the distribution of $\mathcal{I}$.
$\varepsilon$ and $\delta$ are chosen as certain small values such as the machine precision of the computation device, as our objective is to determine whether $v$ is nonzero.
When the distribution of $\mathcal{I}$ is fixed, the lower bound to $k$ can be determined by balancing the accuracy of the hypothesis testing and the number of required iteration, regardless of the choice of nonexpansive operator $\opT$ or the optimization problem (RC-FPI) is solving.
}

For the proof of the \cref{thm:detection}, we begin with the simpler case where $\bv \in \range \left( \opI - \opT \right)$.
Let ${\Mnorm{ \bv} \leq \delta}$ be the null hypothesis with $\delta$ satisfying $\alpha \delta < \epsilon$.
% Then, 
% \[
%  \expec \left[ \Mnorm{\frac{x^k}{k} + \alpha \bv}^2 \right] \lesssim \frac{\left( \beta - \alpha^2\right) \Mnorm{\bv}^2}{k}
% \]
% as $k\rightarrow\infty$.
By the triangle inequality, Markov inequality, and \cref{thm:tv}, under the null hypothesis,
\begin{align*}
\mathbb{P}\left(  \Mnorm{\frac{x^k}{k} } 
\geq \varepsilon \right) 
&\le
\mathbb{P}\left(  \Mnorm{\frac{x^k}{k} + \alpha \bv} 
\geq 
\varepsilon-\alpha\delta
\right)\\ 
&\leq \frac{1}{(\varepsilon-\alpha\delta)^2}\expec \left[ \Mnorm{\frac{x^k}{k} + \alpha \bv}^2 \right]\\
&\lesssim \frac{\left( \beta - \alpha^2\right)\delta^2}{k(\varepsilon-\alpha\delta)^2}
\end{align*}
as $k\rightarrow\infty$.

When $\bv \notin \range (\opI - \opT)$, we can still obtain the same (asymptotic) statistical significance with the same test and the same iteration count $k\gtrsim \frac{\left( \beta - \alpha^2\right) \delta^2}{p\left(\varepsilon - \alpha \delta \right)^2} $. Below, we present the full proof of this general case.
% without the assumption $\bv \in \range (\opI - \opT)$.

\begin{proof}
    First by the triangle inequality and Markov inequality,
\begin{align*}
\mathbb{P}\left(  \Mnorm{\frac{x^k}{k} } 
\geq \varepsilon \right) 
&\le
\mathbb{P}\left(  \Mnorm{\frac{x^k}{k} - \expec \left[ \frac{x^k}{k} \right]} 
\geq 
\varepsilon- \Mnorm{\expec \left[ \frac{x^k}{k} \right]}
\right)\\ 
&\leq \left(\varepsilon- \Mnorm{\expec \left[ \frac{x^k}{k} \right]} \right)^{-2} \mvar\left(  \frac{x^k}{k} \right).
%&\lesssim \frac{\left( \beta - \alpha^2\right)\delta^2}{k(\varepsilon-\alpha\delta)^2}
\end{align*}

Next, let's bound the term $\Mnorm{\expec \left[ {x^k}/{k} \right]}$. Due to triangle inequality, Jensen's inequality and \cref{lem:L2} with $z^0 = x^0$, we have 
\[
\Mnorm{\expec \left[ \frac{x^k}{k} \right]} \le \Mnorm{\expec \left[ \frac{x^k}{k} \right] - \frac{z^k}{k} }  +\Mnorm{ \frac{z^k}{k} } \le \bigO\left( \frac{1}{\sqrt{k}}\right)+\Mnorm{ \frac{z^k}{k} }.
\]
By \citep[Theorem~3]{Applegate2021}, for any $\omega > 0$, there exist $\omega$-dependent constant $C_\omega$ such that
\[
\Mnorm{ \frac{z^k}{k} } \le \alpha \Mnorm{\bv} + \frac{1}{k}C_\omega + \omega.
\]
Thus,
\[
\Mnorm{\expec \left[ \frac{x^k}{k} \right]} 
\le \alpha \Mnorm{\bv} + \bigO\left( \frac{1}{\sqrt{k}}\right)+ \frac{1}{k}C_\omega + \omega.
\]

Substitute this inequality at $\Mnorm{\expec \left[ {x^k}/{k} \right]} $ and obtain
\begin{align*}
k\mathbb{P}\left(  \Mnorm{\frac{x^k}{k} } 
\geq \varepsilon \right) 
&\leq \left(\varepsilon- \Mnorm{\expec \left[ \frac{x^k}{k} \right]} \right)^{-2} k\mvar\left(  \frac{x^k}{k} \right)\\
&\leq \left(\varepsilon- \alpha \Mnorm{\bv} - \bigO\left( \frac{1}{\sqrt{k}}\right)- \frac{1}{k}C_\omega - \omega \right)^{-2} k\mvar\left(  \frac{x^k}{k} \right),
%&\lesssim \frac{\left( \beta - \alpha^2\right)\delta^2}{k(\varepsilon-\alpha\delta)^2}
\end{align*}
when choice of $\omega$ is sufficiently small and that of $k$ is sufficiently large to keep 
\[
\varepsilon- \alpha \Mnorm{\bv} - \bigO\left( \frac{1}{\sqrt{k}}\right)- \frac{1}{k}C_\omega + \omega > 0.
\]

Take a limit supremum by $k\to \infty$. Then by \cref{thm:tv},
\begin{align*}
&\limsup_{k\to \infty} k\mathbb{P}\left(  \Mnorm{\frac{x^k}{k} } \geq \epsilon\right) \\
& \hspace{0.5cm} \leq \limsup_{k\to \infty}\left(\varepsilon- \alpha \Mnorm{\bv} - \bigO\left( \frac{1}{\sqrt{k}}\right)- \frac{1}{k}C_\omega - \omega \right)^{-2} k\mvar\left(  \frac{x^k}{k} \right)\\
& \hspace{0.5cm} \leq  \frac{ (\beta-\alpha^2) \Mnorm{\bv}^2}{\left(\varepsilon- \alpha \Mnorm{\bv} - \omega \right)^2} \\
%&\lesssim \frac{\left( \beta - \alpha^2\right)\delta^2}{k(\varepsilon-\alpha\delta)^2}
\end{align*}
holds for all sufficiently small $\omega>0$. Thus, by $\omega \to 0$,
\[
\limsup_{k\to \infty} k\mathbb{P}\left(  \Mnorm{\frac{x^k}{k} } \geq \epsilon\right)
\leq  \frac{ (\beta-\alpha^2) \Mnorm{\bv}^2}{\left(\varepsilon- \alpha \Mnorm{\bv} \right)^2}
\]

Now consider a null hypothesis of ${\Mnorm{ \bv} \leq \delta}$. Under the null hypothesis,
\[
\limsup_{k\to \infty} k\mathbb{P}\left(  \Mnorm{\frac{x^k}{k} } \geq \epsilon\right)
\leq  \frac{ (\beta-\alpha^2) \delta^2}{\left(\varepsilon- \alpha \delta \right)^2},
\]
or in other word,
\[
\mathbb{P}\left(  \Mnorm{\frac{x^k}{k} } \geq \epsilon\right)
\lesssim  \frac{ (\beta-\alpha^2) \delta^2}{k\left(\varepsilon- \alpha \delta \right)^2},
\]
as $k\rightarrow\infty$.

% for any $c>0$.
% Then, under the null hypothesis,
% \[
% \mathbb{P}\left(  \Mnorm{\frac{x^k}{k} } \geq  c + \alpha \delta \right) \lesssim \frac{1}{kc^2}\left( \beta - \alpha^2\right)\delta^2.
% \]

% Thus, to reject the null hypothesis using the criterion 
% \[
% \Mnorm{\frac{x^k}{k} } > \epsilon
% \]
% with the significance level $p$, 

%%%%%%%%%%%%%%%%%%%%%%%%%%%%%%%%%%%%%%%%%%%%%%%%%%%%%%%%%%%%%%%%%%%%%%%%%%%%%%%

% for any $c>0$.
% Then, under the null hypothesis,
% \[
% \mathbb{P}\left(  \Mnorm{\frac{x^k}{k} } \geq  c + \alpha \delta \right) \lesssim \frac{1}{kc^2}\left( \beta - \alpha^2\right)\delta^2.
% \]

% Thus, to reject the null hypothesis using the criterion 
% \[
% \Mnorm{\frac{x^k}{k} } > \epsilon
% \]
% with the significance level $p$, 

\end{proof}

\section{Extension to non-orthogonal basis and applications to decentralized optimization} 
\label{sec:decent_opt}

% In practice, given operator could be a non-expansive mapping in a non-Euclidean norm.
% In this case, it might be hard to find the orthogonal basis to perform {RC-FPI} or the orthogonal basis of the system might not coincide with the standard basis.
% In such cases, mixture of {FPI} and {RC-FPI} can solve the problem.
% We may perform {RC-FPI} at a subspace with useful orthogonal basis while applying {FPI} on the independent, but possibly not orthogonal, subspace.

Operator splitting methods such as ADMM/DRS \citep{Glowinski1975admm,gabay1976dual,lions1979splitting} or PDHG \citep{chambolle2011first} are fixed-point iterations with operators that are non-expansive with respect to $M$-norms where $M\neq \opI$, and in such cases, the coordinates form a non-orthogonal basis. Our analyses of Sections~\ref{sec:linearrate} and \ref{sec:variance} were mostly general, accommodating any $M$-norm, with the sole exception of \cref{lem:expec_norm}, which only applies to the case where $M=\opI$. In this section, we use the notion of the Friedrichs angle to extend our analysis to general $M$-norms. We then apply our framework to decentralized optimization and present a numerical experiment.

\subsection{Convergence condition in non-orthogonal basis}
Let's modify the underlying space $\cH$ with extra $\cH_0$ block, making $\cH$ as 
\[
\cH=\cH_0 \oplus \cH_1 \oplus \cH_2 \oplus \dots \cH_m,
\]
where each $\cH_i$ is a Hilbert space.
Consider two subspaces $U_1$ and $U_2$ of $\cH$ as
\[
U_1 = \left\{(x_0,0,0,\dots,0) \,\middle|\, x_0 \in \cH_0 \right\},
\quad
U_2 = \left\{(0,x_1,x_2,\dots,x_m) \,\middle|\, x_i \in \cH_i,\, 1\leq i\leq m\right\},
\]
so that $U_1 \cap U_2 = \{0\}$. We further assume that with $M$-inner product of $\cH$, block components of $U_2$ are orthogonal to each other :
\[
\langle (0,0,\dots , x_i, \dots, 0), (0,0,\dots , x_j, \dots, 0) \rangle_M=0,
\quad x_i \in \cH_i,\, x_j \in \cH_j,
\quad 1\leq i <j\leq m.
\]

Note that every vector in $\cH$ can be uniquely expressed as a linear combination of vectors in $U_1$ and $ U_2$.
Given $\opT\colon\cH\to\cH$ and $\opS=\theta^{-1}{(\opI-\opT)}$,
define $\opG$ and $\opH$ as 
\[
\opS x = \opG x + \opH x, \quad \opG x \in U_1, \quad \opH x \in U_2
\]
for all $x\in \cH$. 
We decompose $U_2$ into $m$ block coordinates, which is also the set of orthogonal subspaces.
(To clarify, the $m$ blocks of $U_2$ are orthogonal with respect to the $M$-norm.)
With a selection vector $\ve \in \left[ 0,1 \right]^m$, define a randomized coordinate operator as 
\begin{equation} \label{eq:nonorth_RCFPI}
\begin{split}
    \opS_\ve
    &=
    \alpha \opG + \sum_{i=1}^m \ve_i \opH_i \\
    \opT_\ve
    &=
    \opI - \theta\opS_\ve,
\end{split}
\end{equation}
% \[
% \opS_\ve x= 
% \alpha \opG x+\sum_{i=1}^m \ve_i \opH_i x , \quad \opT_\ve = \opI - \theta \opS_\ve,
% \]
where $\opH_i$ is defined similarly to how $\opS_i$ was defined in \cref{sec:rcupdate}.
% where $\opH_i x$ represents the $i$th block of $\opH x$.

% Furthermore, let's inherit the condition \eqref{condition:P} on $U_2$ and suppose that the distribution of $\ve$ satisfies $\expec_{\ve  } \left[ \ve \right]=\alpha \mathbf{1}$.
% Then, almost sure and $L^2$ convergence of the normalized iterate also holds when the iteration is computed with $\opT_\ve$ for a certain choice of $U_1, U_2$.

The cosine of the Friedrichs angle $c_F$ between $U_1$ and $U_2$ \citep{Friedrichs1937OnCI} is defined as a smallest value among $c\leq 1$ such that satisfies
\[
    \left |\left<u_1, u_2 \right>_M\right| \leq c\Mnorm{u_1}\Mnorm{u_2}
    \quad\forall\,
    u_1 \in U_1,\, u_2 \in U_2.
\]
The RC-FPI by \eqref{eq:nonorth_RCFPI} converges, almost surely and in $L^2$, if the cosine of the Friedrichs angle is sufficiently small.

% Then, almost sure and $L^2$ convergence of the normalized iterate also holds when the iteration is computed with $\opT_\ve$ for a certain choice of $U_1, U_2$.

\begin{theorem} \label{thm:RCandFPI}
Let $\opT\colon\cH\to\cH$ be $\theta$-averaged with respect to $\Mnorm{\cdot}$ with $\theta\in (0,1]$. {Assume that block components of $U_2$ are orthogonal to each other.}
    Let $\bv$ be the infimal displacement vector of $\opT$.
 Assume $\ve^0, \ve^1, \ldots$ is sampled IID from a distribution satisfying the uniform expected step-size condition \eqref{condition:P} with $\alpha\in (0,1]$.
    Let $x^0,x^1,x^2,\ldots$ be the iterates of \eqref{RC-FPI}
    $x^{k+1} = \opT_{\ve^k} x^k$,
    where $\opT_{\ve^k}$ is as defined in \eqref{eq:nonorth_RCFPI}.
    Let $c_F$ be the cosine of the Friedrichs angle between $U_1, U_2$.
    \begin{itemize}
        \item[(a)]
        If $    c_F \leq \sqrt{\frac{1-\theta}{1-\alpha \theta}}$, then ${ x^k} /{k} \stackrel{L^2}{\rightarrow} -\alpha \bv$ as $k\to\infty$.
    \item[(b)]If $c_F < \sqrt{\frac{1-\theta}{1-\alpha \theta}}$, then ${ x^k} /{k} \stackrel{a.s.}{\rightarrow} -\alpha \bv$ as $k\to\infty$. (${ x^k}/ {k}$ converges strongly to $-\alpha \bv$ in probability $1$.)
    Furthermore, {the bias and variance results} of  \cref{thm:tv} hold.
    \end{itemize}
\end{theorem}

\subsection{Proof of \cref{thm:RCandFPI}} \label{appendix:mnorm}

\begin{proposition} \label{lem:f.1}
    Suppose the subspaces $U_1, U_2$ of $\cH$ with
    $ U_1 \cap U_2 = \{ 0 \}$
    satisfy the condition
    \[
    \left |\left<u_1, u_2 \right>_M\right| \leq c_F\Mnorm{u_1}\Mnorm{u_2}, \quad c_F\leq \sqrt{\frac{1-\theta}{1-\alpha \theta}}
    \]
    for any $u_1 \in U_1, u_2 \in U_2$. {Assume that block components of $U_2$ are orthogonal to each other.}
    Then, there exists $\beta\ge0$ such that $\beta\theta \leq \alpha$ and
    \[
    \expec_{\ve  } \left[  u_\ve \right] = \alpha u, 
    \quad \expec_{\ve  } \left[  \Mnorm{u_\ve}^2 \right] \leq \beta \Mnorm{u}^2
    \]
    where $u_\ve$ and $u$ are defined as
    \[
    u_\ve = \alpha g + \sum_{i=1}^m \ve_i (0,\dots , 0, h_i , 0, \dots , 0),
    \quad u = g + h, \quad g \in U_1, \quad h \in U_2. 
    \]
    where $h_i\in \cH_i$ for $i=1,\dots,m$.
\end{proposition}
\begin{proof}
    Let $\tilde{h_i} \in U_2$ as $\tilde{h_i}=(0,\dots , 0, h_i , 0, \dots , 0)$ for $i=1,\dots , m$. First equation comes from, \[
    \expec_{\ve  } \left[  u_\ve \right] = \alpha g + \expec_{\ve  } \left[ \sum_{i=1}^m \ve_i \tilde{h_i} \right] = \alpha g +  \alpha h = \alpha u.
    \]

    The expectation in the second equation is
\begin{equation*}
    \begin{split}
        \expec_{\ve  } \left[  \Mnorm{u_\ve}^2 \right] 
        &= \alpha^2 \Mnorm{g}^2 +2\expec \left[ \left< \sum_{i=1}^m \ve_i \tilde{h_i}, \alpha g\right>_M\right] +\expec \left[ \Mnorm{\sum_{i=1}^m \ve_i \tilde{h_i} }^2\right] \\
        &=\alpha^2 \Mnorm{g}^2
        +2 \left< \alpha g, \alpha h\right>_M
        +\sum_{i=1}^m  \expec \left[ \ve_i^2\right]  \Mnorm{ \tilde{h_i}}^2
         \\
         &\leq\alpha^2 \Mnorm{g}^2
        +2 \left< \alpha g, \alpha h\right>_M
        +\sum_{i=1}^m  \expec \left[ \ve_i\right]  \Mnorm{ \tilde{h_i}}^2
         \\
        &=\alpha^2 \Mnorm{g}^2
        +2 \alpha^2 \left<  g,  h\right>_M
        + \alpha \Mnorm{h}^2 \\
        &= \alpha^2 \Mnorm{u}^2 + (\alpha - \alpha^2) \Mnorm{h}^2.
    \end{split}
\end{equation*}
{In the second equality, the Pythagorean theorem is applied on $\sum_{i=1}^m \ve_i \tilde{h_i}$, since the block components of $U_2$ are orthogonal to each other.} Note that 
\[
\Mnorm{u}^2
= \Mnorm{h}^2 +2\left< h, g\right>_M+ \Mnorm{g}^2
\geq \Mnorm{h}^2 -2c_F\Mnorm{h}\Mnorm{g}+ \Mnorm{g}^2
\geq \left( 1-c_F^2\right) \Mnorm{h}^2.
\]
Thus, set $\beta$ as
\[
\beta = \alpha^2 + \frac{\alpha - \alpha^2}{1-c_F^2}.
\]
Then,
\[
\expec_{\ve  } \left[  \Mnorm{u_\ve}^2 \right] \leq  \alpha^2 \Mnorm{u}^2 + (\alpha - \alpha^2) \Mnorm{h}^2
\leq \left( \alpha^2 + \frac{\alpha - \alpha^2}{1-c_F^2}\right) \Mnorm{u}^2 =\beta  \Mnorm{u}^2,
\]
with
\[
\theta \beta \leq \theta \left( \alpha + \frac{1-\alpha\theta}{ \theta}\right) \alpha  = \alpha.
\]
\end{proof}
Additionally, if $c_F< \sqrt{\frac{1-\theta}{1-\alpha \theta}} $ in \cref{lem:f.1}, we have $\theta\beta < \alpha$.

\begin{proof} [Proof of \cref{thm:RCandFPI}]
With \cref{lem:f.1}, we know that $\beta$ is dependent on the value of the cosine of Friedrichs angle $c_F$ as :
\[
\expec_{\ve  } \left[  u_\ve \right] = \alpha u, 
\quad \expec_{\ve  } \left[  \Mnorm{u_\ve}^2 \right] 
\leq \beta \Mnorm{u}^2,
\quad \beta = \alpha^2 + \frac{\alpha - \alpha^2}{1-c_F^2}.
\]
Hence, when $c_F \leq \sqrt{\frac{1-\theta}{1-\alpha \theta}}$, we have $\beta \leq \alpha / \theta$, and when $c_F < \sqrt{\frac{1-\theta}{1-\alpha \theta}}$, we have $\beta < \alpha / \theta$.

\textbf{Proof of statement (a).} Since $c_F \leq \sqrt{\frac{1-\theta}{1-\alpha \theta}}$, we have $\beta \leq \alpha / \theta$. Therefore, we may use the result of  \cref{lem:L2} with $z^0 = x^0$.
\[
\expec\left[ \Mnorm{\frac{ x^k} {k} - \frac{ z^k} {k} }^2 \right]
\leq 
\frac{1} {k} \left( 
2\sqrt{\alpha \theta}  \left( 1-\alpha \theta \right) \Mnorm{ \opS x^{0}  } \Mnorm{\opS z^0}
- \frac{\alpha}{\theta} \left(1 - \alpha \theta \right)\Mnorm{\bv}^2
\right) .
    \]

When the limit $k\to \infty$ is taken, 
    \[
\lim_{k\to \infty} \expec\left[ \Mnorm{\frac{ x^k} {k} - \frac{ z^k} {k} }^2 \right]=0, \quad \lim_{k\to\infty} \Mnorm{\frac{ z^k} {k} +\alpha \bv}=0,
    \]
    where the second equation is from \cref{thm:pazy}. These two limits provide $L^2$ convergence of normalized iterate, namely
    \[
 \frac{ x^k} {k}  \stackrel{L^2}{\rightarrow} -\alpha \bv,
    \]
    as $k\to \infty$.

\textbf{Proof of statement (b).}  Since $c_F < \sqrt{\frac{1-\theta}{1-\alpha \theta}}$, we have $\beta < \alpha / \theta$. Thus, from \cref{lem:as}, we can conclude the strong convergence in probability $1$,
    \[
    \frac{ x^k} {k} \stackrel{a.s.}{\rightarrow} -\alpha \bv
    \]
    as $k\to\infty$.
    Furthermore, since $\beta < \alpha / \theta$, we now satisfy every conditions of \cref{thm:tv}. Thus, identical results of \cref{thm:tv} are obtained in this case.
\end{proof}

\subsection{Application of \cref{thm:RCandFPI} in {PG-EXTRA}} \label{subsec:PGEXTRA}
Consider the convex optimization problem
\begin{equation}
\label{eq:problem}
\begin{array}{ll}
\underset{x\in \mathbb{R}^d}{\mbox{minimize}}
&
\displaystyle{
\sum_{i=1}^m f_i(x)},
\end{array}
\end{equation}
where $f_i\colon\RR^d \to \RR$ is closed, convex, and proper function for $i=1,\dots,m$.
Consider the decentralized algorithm {PG-EXTRA} \citep{shi2015proximal}
% \begin{equation}
% \label{iter:PG-EXTRA}
% \begin{split}
%     x_i^{k+1} 
%     &= \prox_{\alpha f_i}\left(\sum_{j=1}^m W_{ij}x_i^{k} - w_i^{k} \right) \\
%     w_i^{k+1} &= w_i^{k}
%     + \frac{1}{2} \left( x_i^k - \sum_{j=1}^m W_{ij} x_j^k \right)
% \end{split}
% \end{equation}
\begin{equation*}
    \label{iter:PG-EXTRA} \tag{PG-EXTRA}
    \begin{split}
        x_i^{k+1} 
    &= \prox_{\tau  f_i}\left(
        {\scriptstyle \sum_{j=1}^m} W_{ij} x_i^{k} - w_i^{k}
    \right) \\
    w_i^{k+1} &= w_i^{k}
    + \frac{1}{2} \left( x_i^k - {\scriptstyle \sum_{j=1}^m} W_{ij} x_j^k \right)
    \end{split}
\end{equation*}
for $i=1,2,\dots , m$. In decentralized optimization, we use network of agents to compute the algorithm. If a pair of agents could communicate, we say that they are connected. 
For each agents $i=1,2\dots ,m$, $N_i$ is a set of agents connected to agent $i$.
% $W$ is a mixing matrix.
A matrix $W$ is called a mixing matrix, and it is a symmetric $m$ by $m$ matrix with $W_{ij} = 0$ if $i\neq j$ and $j \notin N_i$.
{Furthermore, $1\geq |\lambda_i|$ for each eigenvalue of $W$ and $\ker(I-W) = \text{span}(\mathbf{1})$, which ensures consensus among agents at the solution. }

A randomized coordinate-update version of {PG-EXTRA} randomly chooses $i$ among $1, 2, \dots ,m$ to update $x_i^k$, while every $w_1, w_2, \dots , w_m$ gets updated at each iterations.

\begin{algorithm}[htbp]
   \caption{RC-PG-EXTRA}
   \label{alg:PG-EXTRA}
\begin{algorithmic}
   \For{$i\in \{1,2,\dots, m\}$}
   \State {\bfseries Initialize:} $w_i=0$, $x_i=0$, $[Wx]_i=0$
   \EndFor

   \For{$j\in \{1,2,\dots, m\}$}
   \State {\bfseries Update:} $w_j =  w_j + \frac{\alpha}{2} \left(x_i- [Wx]_i \right)$
   \EndFor
   
   \While{Not converged}
   \State {\bfseries Sample:} $\ve $
   \For{$i$ such that $\ve_i \neq 0$}
   
   \State $\Delta x_i = \prox_{\tau  f_i} \left([Wx]_i - w_i \right) -x_i$
   \State {\bfseries Update:} $x_i = x_i + \ve_i \Delta x_i$

   \For{$j\in N_i \cup \{i\}$}
   \State {\bfseries Send:} {$\Delta x_i$ From $i$th agent to $j$th agent.}
   \State $[Wx]_j = [Wx]_j + W_{ij} \Delta x_i$
   \State $w_j=  w_j + \frac{\alpha}{2} \left(x_i- [Wx]_i \right)$
   \EndFor

   \EndFor
   
   \EndWhile
\end{algorithmic}
\end{algorithm}

Note that $\Delta x_i$ is the only quantity communicated across agents.
% that's being communicated by agents. 
$|N_i|$ communications happen each iteration, while values $x_i, w_i, [Wx]_i$ are stored in $i$th agent.

\eqref{iter:PG-EXTRA} is a fixed-point iteration with an averaged operator with respect to $M$-norm where $M\neq \opI$.
Under the conditions of \cref{thm:PG-EXTRA}, the condition regarding the Friedrichs angle of \cref{thm:RCandFPI} holds and \cref{alg:PG-EXTRA} converges.

% when a certain condition on the mixing matrix, which implies the condition about Friedrichs angle in \cref{thm:RCandFPI}, is satisfied. The following corollary presents the condition on the mixing matrix to garantee the convergence of \cref{alg:PG-EXTRA}.

% Therefore, \cref{thm:RCandFPI} guarantees the convergence of the randomized coordinate-update version of {PG-EXTRA}, \cref{alg:PG-EXTRA}, 
% when a certain condition on the mixing matrix, which implies the condition about Friedrichs angle in \cref{thm:RCandFPI}, is satisfied. The following corollary presents the condition on the mixing matrix to garantee the convergence of \cref{alg:PG-EXTRA}.

\begin{corollary} \label{thm:PG-EXTRA}
Suppose $\ve^0, \ve^1, \ldots$ is sampled IID from a distribution satisfying the uniform expected step-size condition \eqref{condition:P} with $\alpha\in (0,1]$. 
Consider \cref{alg:PG-EXTRA} with $\ve = \ve^0, \ve^1, \ldots$.
If the minimum eigenvalue of the symmetric mixing matrix $W\in\mathbb{R}^m$ satisfies
    \[
    % \lambda_1, \lambda_2, \dots , \lambda_m 
    \lambda_\text{min}(W)> -\frac{\alpha}{2-\alpha},
    \]
the normalized iterate of \cref{alg:PG-EXTRA} converges to $-\alpha \bv$, where $\bv$ is the infimal displacement vector of \eqref{iter:PG-EXTRA}, both in $L^2$ and almost surely.
\end{corollary}

\begin{proof}
    
In the proofs, we use a stack notation for convenience. With stack notation, $\vx\in\RR^{m\times d}$ refers
\[
\vx = 
\begin{bmatrix}
    \text{---} & x_1^\intercal & \text{---} \\
    \text{---} & x_2^\intercal & \text{---} \\
    & \vdots & \\
    \text{---} & x_m^\intercal & \text{---}
\end{bmatrix}, \quad
\left[ W \vx \right]_i = \sum_{j=1}^m W_{ij}x_j .
\]

\eqref{iter:PG-EXTRA} originates from {Condat-V\~u} \citep{condat2013primal,vu2013splitting} with $\vw^k = \tau U \vu^k$, where {Condat-V\~u} is a method defined as 
\begin{equation*}
\begin{split}
    &\vx^{k+1} = \prox_{\tau  f} (W\vx^k - \tau  U \vu^k ) \\
    &\vu^{k+1} = \vu^k + \frac{1}{\tau } U \vx^k,
\end{split}
\end{equation*}
which is a fixed-point iteration with an operator that's {$({1}/{2})$-averaged} in $M$-norm. Thus, $\theta$ value in \eqref{iter:PG-EXTRA} is $\theta = {1}/{2}$.

The matrix $M$ in \eqref{iter:PG-EXTRA} is
\begin{equation*}
    \begin{split}
        M
        = 
        \begin{bmatrix}
        \frac{1}{\tau } I & U \\
        U & \tau  I
        \end{bmatrix}, 
    \end{split}
\end{equation*}
where $U$ is a positive semidefinite matrix such that $U^2 = \frac{1}{2}\left( I-W\right)$. Note that the inner product in this case is
\[
\left< 
\begin{bmatrix}
        \vx \\
        \vu
        \end{bmatrix},
\begin{bmatrix}
        \vy \\
        \vv
        \end{bmatrix}\right>_M
        = \tr \left( 
        \begin{bmatrix}
        \vx \\
        \vu
        \end{bmatrix}^T 
        M \begin{bmatrix}
        \vy \\
        \vv
        \end{bmatrix} \right).
\]
Due to the given inner product, two subspaces $V_1 = \left( \RR^m \times \{0\}^m\right)^d$ and $V_2 = \left(  \{0\}^m\times \RR^m  \right)^d$ are no longer orthogonal to each other. On the other hand, $m$ subspaces of $V_1$, 
\[
\left(  \{0\}^{i-1} \times \RR \times \{0\}^{m-i} \times \{0\}^m\right)^d, \quad i=1,2,\dots , m,
\]
are orthogonal to each other. Inner product between $V_1$ and $V_2$ is constrained as
\[
\left|
\left< 
\begin{bmatrix}
        \vx \\
        \mathbf{0}
        \end{bmatrix},
\begin{bmatrix}
        \mathbf{0} \\
        \vu
        \end{bmatrix}\right>_M \right|
        = \left| \vx^T U \vu \right| \leq \lambda_{\max}^U \norm{\begin{bmatrix}
        \vx \\
        \mathbf{0}
        \end{bmatrix}} \norm{\begin{bmatrix}
        \mathbf{0} \\
        \vu
        \end{bmatrix}}
        = \lambda_{\max}^U \Mnorm{\begin{bmatrix}
        \vx \\
        \mathbf{0}
        \end{bmatrix}} \Mnorm{\begin{bmatrix}
        \mathbf{0} \\
        \vu
        \end{bmatrix}}.
\]
Since $\lambda_{\min}^W> -{\alpha}/({2-\alpha})$, 
\[
\lambda_{\max}^U = \sqrt{\frac{1-\lambda_{\min}^W}{2}} < \sqrt{\frac{1}{2-\alpha}} = \sqrt{\frac{1-\frac{1}{2}}{1-\frac{\alpha}{2}}},
\]
and we may apply \cref{thm:RCandFPI} with
\[
\vu \in V_2 = U_1, \quad \vx \in V_1 = U_2, \quad \cH_0 = \RR^{m \times d}, \quad \cH_1=\cH_2=\dots = \cH_m = \RR^d
\]
and block coordinate update with each orthogonal blocks as 
\[
\left(  \{0\}^{i-1} \times \RR \times \{0\}^{m-i} \times \{0\}^m\right)^d, \quad i=1,2,\dots , m,
\]
to conclude \cref{thm:PG-EXTRA}.

\end{proof}

Additionally, here is the infimal displacement vector of \eqref{iter:PG-EXTRA}.
\begin{lemma} \label{lem:v_pgextra}
    The infimal displacement vector $\bv=(\bv_1,\dots,\bv_m)$ of \eqref{iter:PG-EXTRA} is
    \[
\bv_i = 
\begin{bmatrix}
    \frac{\tau }{m} \sum_{j=1}^m g_j \\
    -\frac{1}{2} \left( y_i - \sum_{j=1}^m W_{ij}y_j \right)
\end{bmatrix}
\]
for $i=1,\dots,m$, where $\left( y_1, y_2, \dots , y_m\right)$ and $\left( g_1, g_2, \dots , g_m\right)$ are
\[
 \underset{
 \begin{subarray}{c}
 y_1, y_2, \dots y_m \in \RR^{d} \\
  g_j \in \partial f_j(y_j), 1\leq j \leq m
  \end{subarray}
  }{\argmin} ~    \frac{\tau^2 }{m}\norm{ \sum_{j=1}^m g_j}^2 + \frac{1}{4} \sum_{i,j=1}^m W_{ij} \norm{y_i - y_j}^2.
\]
\end{lemma}
{
% \color{blue}
The proof of \cref{lem:v_pgextra} is presented in the \ref{appendix:pgextra}.
}

\subsection{Experiment of the infeasible case in PG-EXTRA} \label{sec:exp}

{
In this subsection, we show a problem instance for which \eqref{RC-FPI} is efficient.
}
We performed an experiment on an instance of 
\eqref{eq:problem} using \cref{alg:PG-EXTRA}. 
{We considered a setup where a single abnormal agent makes the whole decentralized optimization problem infeasible.}
\cref{fig:experiment} shows that (RC-PG-EXTRA), \cref{alg:PG-EXTRA}, converges to the  infimal displacement vector faster in terms of communication count.

% shows the distance between the normalized iterate of \cref{alg:PG-EXTRA} and the

\begin{figure*}[ht]
\centering
\includegraphics[width=0.45\textwidth]{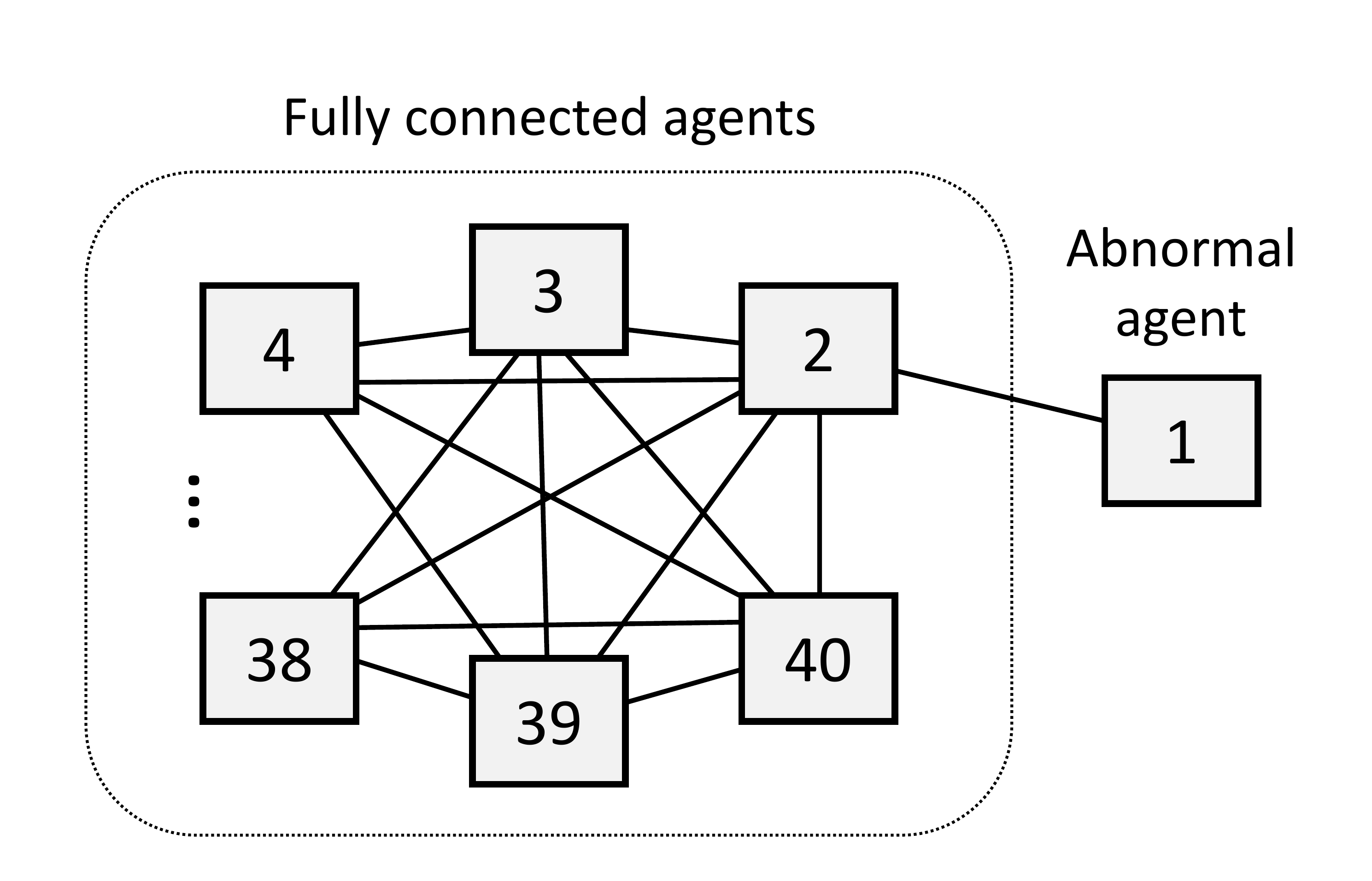}
\hspace{0.05\textwidth}
\includegraphics[width=0.45\textwidth]{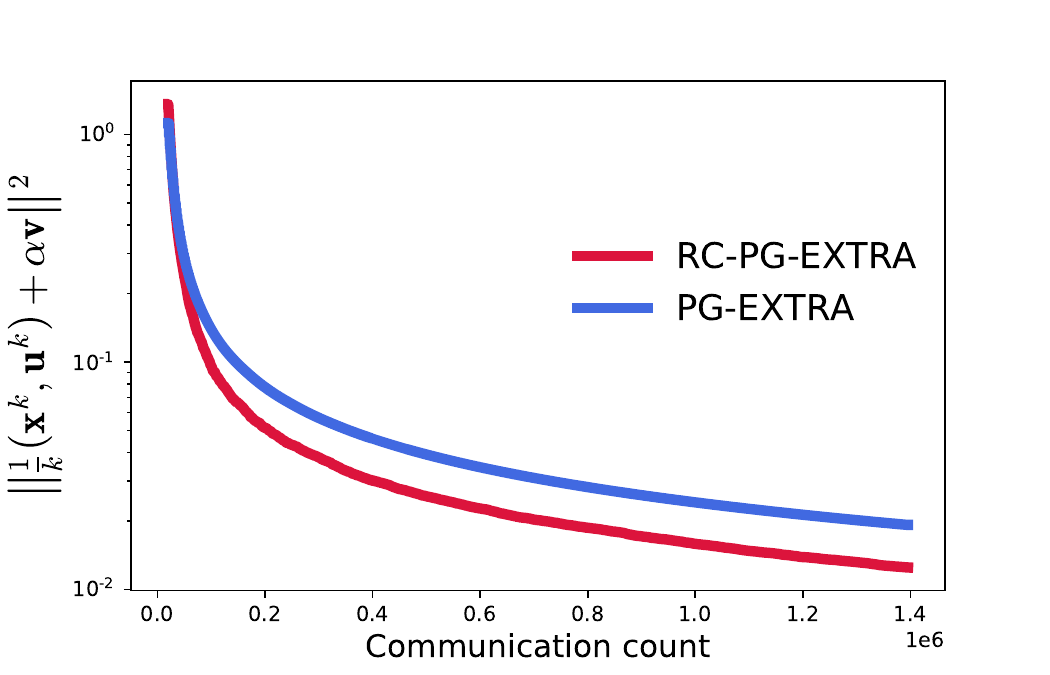}
  % \subfigure[]{\includegraphics[width=0.45\textwidth]{./net.pdf}\label{fig:network}}
  % \hspace{.05\textwidth}
  % \subfigure[]{\includegraphics[width=0.45\textwidth]{./PG-Extra.pdf}\label{fig:result}} 
  \caption{(Left) Network used in our experiment, consisting of $m=40$ agents, with agents $2,\dots,40$ densely connected.
  (Right) Graph of $\norm{\left( \mathbf{x}^k,  \mathbf{u}^k \right) /k + \alpha \bv}^2$ against the communication count for \eqref{iter:PG-EXTRA} and (RC-PG-EXTRA), \cref{alg:PG-EXTRA}.}
  \label{fig:experiment}
\end{figure*}

% $m-1$ agents share overlapping domain with fully connected network while the remaining agent has a non-overlapping domain to the others, causing the infeasibility of the problem. With such setting, it is possible to compute explicit value of infimal displacement vector $\bv$. 

% Consider the inconsistent decentralized optimization setting where one abnormal agent (abnormality) causes infeasibility of the problem.
% We gave this situation by setting $m-1$ agents, fully connected, to solve a feasible problem, while one extra agent, with non-overlapping domain, to be connected with one of the $m-1$ agents. This network is visualized at \cref{fig:network}. We have given such situation because it is possible to compute explicit value infimal displacement vector $\bv$. 

% We chose the mixing matrix $W$ as a \emph{Metropolis mixing matrix} \textcolor{red}{(XXXreference?)}

% In this setup, $m-1$ agents have constraints that overlap, while the $1$ remaining agent has an individual constraint that does not overlap with other agents' constraints.
Specifically, define $f_i \colon \RR^2 \to \RR$ for $i=1,\cdots,m$ as
\[
f_i(x)=
\left\{
\begin{array}{ll}
0&\text{if }x \in C_i\\
\infty&\text{otherwise.}
\end{array}
\right.
\]
with $C_1 = \left\{ (x,y) \mid x \leq -10 \right\}$ and  $C_2 = C_3 = \dots=C_m  = \left\{ (x,y) \mid x > 0, \, xy \leq -1 \right\}$.
% \begin{equation*}
% f_i = 
% \begin{cases}
%  \delta_{C_1}  & \text{if $i=1$} \\
%  \delta_{C_2} & \text{if $i \geq 2$.}
%   \end{cases}
% \end{equation*}
% where $\delta_K$ is an indicator function of set $K$ and
% \begin{equation*}
%     \begin{split}
%         & C_1 = \left\{ (x,y) \mid x \leq -10 \right\} \\
%         & C_2 = \left\{ (x,y) \mid x > 0, \, xy \leq -1 \right\}.\\
%     \end{split}
% \end{equation*}
The network is depicted in \cref{fig:experiment}.
We use Metropolis constant edge weight matrix  \citep{boyd2004fastest,xiao2007distributed} for our mixing matrix $W$. Metropolis mixing matrix is a symmetric matrix of the form
\begin{equation*}
    W_{ij} = \begin{cases}
  \frac{1}{\max \left( |N_i|, |N_j| \right) + \epsilon }  & \text{if $j\in N_i$} \\
  1-\sum_{l\in N_i} W_{il} &  \text{if $j = i$} \\
  0 & \text{otherwise}
\end{cases}
\end{equation*}
with $\epsilon>0$.
We choose $\epsilon=0.05$ in our experiment.

In this setting, the infimal displacement vector has the analytical form:
\begin{equation*}
\bv_i = \frac{b_i}{2(m-1+\epsilon)} \begin{bmatrix} \mathbf{0} \\u_1 - u_2 \end{bmatrix}, \quad 
b_i = \begin{cases}
  1& \text{if $i=1$} \\
  -1 &  \text{if $i = 2$} \\
  0 & \text{if $i>2$,}
\end{cases}
\end{equation*}
where $u_1, u_2\in \RR^d$ is a vector defined as {
\[
(u_1, u_2) = \argmin_{u_1\in\overline{C_1}, u_2\in\overline{C_2}} \|u_1 - u_2\|.
\]
}{
The detailed computation of infimal displacement vector is presented in \ref{appendix:infdistvec}.
}

The distribution of $\ve$ used for the experiment is
\begin{equation*}
P\left( \ve \right) = \begin{cases}
 0.3 & \text{if $\ve = \frac{0.7}{0.3 \times (m-1)} e_1$} \\
  \frac{0.7}{m-1} &  \text{if $\ve = e_i$ for some $i\geq 2$} \\
  0 & \text{otherwise,}
\end{cases}
\end{equation*}
where $e_i\in\mathbb{R}^m$ is the $i$th standard unit vector.

% Number of agents $m$ is $40$, size of dimension $d$ is $2$, and $\epsilon=0.05$. Convex sets $C_1, C_2$ that determine $f_i$'s are chosen as
% \begin{equation*}
%     \begin{split}
%         & C_1 = \left\{ (x,y) \mid x \leq -10 \right\} \\
%         & C_2 = \left\{ (x,y) \mid x > 0, \, xy \leq -1 \right\}.\\
%     \end{split}
% \end{equation*}

% Then, the distance between the infimal displacement vector and the normalized iterate per communication between agents follows a graph of \cref{fig:result}.

\section{Conclusion} \label{sec:conclusion}
This work analyzes the asymptotic behavior of the \eqref{RC-FPI} and establishes convergence of the normalized iterates to the infimal displacement vector, and this allows us to use the normalized iterates to test for infeasibility. We also extend our analyses to the setup with non-orthogonal basis, thereby making our results applicable to the decentralized optimization algorithm \eqref{iter:PG-EXTRA}.

One possible direction of future work would be to use variance reduction techniques in the style of, say, SVRG \citep{NIPS2013_ac1dd209} or \citep{schmidt2017} to improve the convergence rate. Such techniques allow stochastic-gradient-type methods to exhibit a rate faster than $\mathcal{O}(1/k)$, and may be applicable in to the coordinate-update setup accelerate the infeasibility detection.

{
Another possible direction is to generalize our (RC-FPI) framework to cover the coordinate-update primal-dual methods of Chambolle--Ehrhardt--Richt\'arik--Sch\"onlieb \cite{ChambolleEhrhardtRichtarikSchonlieb2018_stochastic} and Fercoq--Bianchi \cite{FercoqBianchi2019_coordinatedescent}, which show faster convergence speeds for feasible optimization problems. We believe that the practical usability of these methods can be enhanced with additional infeasibility detection mechanisms similar to those presented in this paper.
}

% \textcolor{blue}{
% Other possible future work would be further extension to the stochastic methods that does not follow the format of \eqref{RC-FPI} of this paper. Primal-dual algorithms with hybrid coordinate updates 
% \citep{ChambolleEhrhardtRichtarikSchonlieb2018_stochastic, FercoqBianchi2019_coordinatedescent} show faster convergence speed in practice. Such algorithms are incompatible with the theorems of this paper since the uniform expected step-size condition \eqref{condition:P} does not hold in these cases. However, with additional assumptions such as coordinate-wise non-expansiveness and ESO inequality used in these algorithms, it may be possible to develop similar results of this paper.
% }

\section*{Acknowledgments}
 We thank Kibeom Myoung for providing careful reviews and valuable feedback. 
 This work is supported by the National Research Foundation of Korea (NRF) grants funded by the Korean government (No.RS-2024-00421203) and (RS-2024-00406127).

 % The authors have been funded by the National Research Foundation of Korea (NRF) Grant funded by the Korean Government (MSIP) [grant number 2020R1F1A1A01072877], the National Research Foundation of Korea (NRF) Grant funded by the Korean Government (MSIP) [grant number 2017R1A5A1015626], and the Samsung Science and Technology Foundation [Project Number SSTF-BA2101-02]. 








% \end{thebibliography}

\appendix
\section{Omitted proofs of  \cref{sec:linearrate}} \label{appendix:4}
\subsection{Proofs of the lemmas in \cref{subsec:props}} \label{appendix:4.1}
\begin{proof} [Proof of  \cref{lem:reduce_expec}]
    Substitute $\opT_{\ve} = \opI - \theta \opS_{\ve}$ at $\mathop{\expec}_{\ve,X,Y} \left[ \Mnorm{  \opT_{\ve} X - \opT_{\ve} Y }^2\right]$ and apply \eqref{eq:beta-def} with $u$ as $\opS X - \opS Y$ to get
\begin{equation*}
\begin{split}
&\mathop{\expec}_{\ve,X,Y} \left[ \Mnorm{  \opT_{\ve} X - \opT_{\ve} Y }^2\right] = \mathop{\expec}_{\ve,X,Y} \left[ \Mnorm{  X-Y - \theta \left( \opS_{\ve} X - \opS_{\ve} Y\right) }^2\right] \\
& \hspace{0.5cm}  = \mathop{\expec}_{X,Y} \left[ \Mnorm{  X-Y }^2\right] +  \theta^2 \mathop{\expec}_{\ve,X,Y} \left[ \Mnorm{    \opS_{\ve} X - \opS_{\ve} Y }^2\right] 
- 2\theta \mathop{\expec}_{\ve,X,Y} \left[ \left<  X-Y ,  \opS_{\ve} X - \opS_{\ve} Y \right>_M\right] \\
& \hspace{0.5cm}  \leq \mathop{\expec}_{X,Y} \left[ \Mnorm{  X-Y }^2\right] +  \beta \theta^2 \mathop{\expec}_{X,Y} \left[ \Mnorm{   \left( \opS X - \opS Y\right) }^2\right] 
- 2\alpha\theta \mathop{\expec}_{X,Y} \left[ \left<  X-Y ,   \opS X - \opS Y \right>_M\right] .
\end{split}
\end{equation*}
Since $\beta \theta \leq \alpha$ and $\opS$ is {$({1}/{2})$-{cocoercive}},
\begin{equation*}
\begin{split}
&\beta \theta^2 \mathop{\expec}_{X,Y} \left[ \Mnorm{ \opS X - \opS Y }^2\right] \leq \alpha \theta \mathop{\expec}_{X,Y} \left[ \Mnorm{    \opS X - \opS Y }^2\right]
 \leq 2\alpha \theta \mathop{\expec}_{X,Y} \left[ \left<X-Y,  \opS X - \opS Y \right>_M\right].
\end{split}
\end{equation*}
Thus, we can reach the conclusion
\[
\mathop{\expec}_{\ve,X,Y} \left[ \Mnorm{  \opT_{\ve} X - \opT_{\ve} Y }^2\right] 
  \leq  \mathop{\expec}_{X,Y} \left[ \Mnorm{  X-Y }^2\right].
\]
\end{proof}
\begin{proof} [Proof of \cref{lem:L2_bound}]
    First, substitute $\opT_\ve = \opI - \theta \opS_\ve$ and $\Bar{\opT} = \opI - \alpha \theta \opS$ at the expectation 
    $\expec_{\ve }\left[ \Mnorm{\opT_{\ve} x - \Bar{\opT}z }^2\right]$. 

    \begin{equation*}
        \begin{split}
            &\expec \left[ \Mnorm{\opT_{\ve} x - \Bar{\opT}z }^2\right] 
            =\expec \left[ \Mnorm{x-z - \theta \left( \opS_{\ve} x - \alpha \opS z \right) }^2\right] \\
            & \hspace{0.5cm} = \Mnorm{x-z}^2 
            + \theta^2 \expec \left[ \Mnorm{  \opS_{\ve} x - \alpha \opS z  }^2\right]
            -2\theta  \expec \left[ \left< x-z,  \opS_{\ve} x - \alpha \opS z  \right>_M\right] \\
            & \hspace{0.5cm} = \Mnorm{x-z}^2 
            + \theta^2 \expec \left[ \Mnorm{  \opS_{\ve} x - \alpha \opS z  }^2\right]
            -2 \alpha \theta   \left< x-z,  \opS x -  \opS z  \right>_M. 
        \end{split}
    \end{equation*}
    Then, use {$({1}/{2})$-{cocoercive}} property of the operator $\opS$. 
    \begin{equation*}
        \begin{split}
            \expec \left[ \Mnorm{\opT_{\ve} x - \Bar{\opT}z }^2\right] 
            &\leq \Mnorm{x-z}^2 
            + \theta^2 \expec \left[ \Mnorm{  \opS_{\ve} x - \alpha \opS z  }^2\right]
            - \alpha \theta   \Mnorm{\opS x - \opS z}^2. 
        \end{split}
    \end{equation*}
    Finally, apply an inequality 
    \begin{equation*}
        \begin{split}
            &\expec \left[ \Mnorm{  \opS_{\ve} x - \alpha \opS z  }^2\right] 
            = \expec \left[ \Mnorm{  \left( \opS_{\ve} x -\alpha \opS x\right) + \alpha \left( \opS x - \opS z \right)  }^2\right] \\
            &\hspace{0.5cm}= \expec \left[ \Mnorm{\opS_{\ve} x -\alpha \opS x}^2 \right]
            + 2\alpha \left< \expec\left[ \opS_{\ve} x -\alpha \opS x \right], \opS x - \opS z  \right>_M
            + \alpha^2 \Mnorm{\opS x - \opS z }^2\\
            &\hspace{0.5cm}= \expec \left[ \Mnorm{\opS_{\ve} x }^2 \right] - \Mnorm{\alpha \opS x}^2
            + \alpha^2 \Mnorm{\opS x - \opS z }^2\\
            &\hspace{0.5cm}\leq \left( \beta - \alpha^2 \right) \Mnorm{\opS x}^2 + \alpha^2 \Mnorm{\opS x - \opS z}^2,
        \end{split}
    \end{equation*}
    we get the desired inequality 
    \[
\expec_{\ve }\left[ \Mnorm{\opT_{\ve} x - \Bar{\opT}z }^2\right] 
\leq \Mnorm{x-z}^2-\alpha\theta \left( 1-\alpha \theta \right) \Mnorm{\opS x - \opS z}^2 
 + \theta^2 \left( \beta - \alpha^2\right) \Mnorm{\opS x}^2.
    \]
\end{proof}

\subsection{Proofs of the lemmas in \cref{appendix:L2}} \label{appendix:4.2}
\begin{proof} [Proof of \cref{lem:bound_expec}]
Apply \cref{lem:reduce_expec} repeatedly, we get
\[
\expec \left[ \Mnorm{\opT_{\ve^k}  \dots \opT_{\ve^1} X - \opT_{\ve^k} \dots \opT_{\ve^1} Y}^2 \right]
\leq \expec \left[ \Mnorm{ X - Y}^2 \right],
\]
for arbitrary random variable $X, Y$. From Jensen's inequality, 
\[
 \Mnorm{ \expec \left[\opT_{\ve^k}  \dots \opT_{\ve^1} X - \opT_{\ve^k} \dots \opT_{\ve^1} Y \right]}^2
\leq \expec \left[ \Mnorm{ X - Y}^2 \right].
\]
Now set up $X, Y$ as 
\[
X = \opT_{\ve^0} x^0, \quad Y = x^0.
\]
Then as a result, we have an inequality
\[
\Mnorm{ \expec \left[\opT_{\ve^k}  \dots \opT_{\ve^1} \opT_{\ve^0} x^0 - \opT_{\ve^k} \dots \opT_{\ve^1} x^0 \right]}^2
\leq \expec \left[ \Mnorm{ \theta \opS_{\ve^0} x^0}^2 \right] \leq  \beta \Mnorm{\theta \opS x^0}^2.
\]
Since $\ve_0, \ve_1, \dots , \ve_n$ are independent and identically distributed, the following equivalence holds.
\[
\expec \left[\opT_{\ve^k}  \dots \opT_{\ve^1} x^0 \right]
=\expec \left[ \opT_{\ve^ {k-1}} \dots \opT_{\ve^0} x^0 \right].
\]
This equality gives the conclusion
\begin{equation*}
    \begin{split}
        &\Mnorm{ \alpha \expec \left[\theta \opS \opT_{\ve^ {k-1}} \dots \opT_{\ve^1} \opT_{\ve^0} x^0  \right]}\\
&\hspace{0.5cm}=\Mnorm{  \expec \left[\left( \opI - \Bar{\opT}\right) \opT_{\ve^ {k-1}} \dots \opT_{\ve^1} \opT_{\ve^0} x^0  \right]} \\
&\hspace{0.5cm}=\Mnorm{  \expec \left[\left( \opI - \opT_{\ve^k}\right) \opT_{\ve^ {k-1}} \dots \opT_{\ve^1} \opT_{\ve^0} x^0  \right]} \\
&\hspace{0.5cm}= \Mnorm{ \expec \left[\opT_{\ve^k}  \dots \opT_{\ve^1} \opT_{\ve^0} x^0 - \opT_{\ve^{k-1}} \dots \opT_{\ve^0} x^0 \right]}\\
&\hspace{0.5cm}= \Mnorm{ \expec \left[\opT_{\ve^k}  \dots \opT_{\ve^1} \opT_{\ve^0} x^0 - \opT_{\ve^k} \dots \opT_{\ve^1} x^0 \right]}\leq  \beta^{1/2}  \Mnorm{\theta \opS x^0}.
    \end{split}
\end{equation*}
\end{proof}

\begin{proof}[Proof of \cref{lem:bound_FPI}]
    All we need to prove is,
    \[
\Mnorm{\opS \opT z} \leq \Mnorm{\opS z}, \qquad \forall z \in \cH.
    \] 
    From $\opS$ being {$({1}/{2})$-{cocoercive}} operator,
    \[
2 \left< \opS \opT z - \opS z, \opT z - z \right>_M\geq  \Mnorm{\opS \opT z - \opS z}^2.
    \]
    With $\opT z - z = -\theta \opS z$, we get
    \[
 \theta \left< \opS \opT z - \opS z, - \opS z - \opS \opT z \right>_M\geq (1-\theta) \Mnorm{\opS \opT z - \opS z}^2\geq 0,
    \]
    which is equivalent to 
        \[
\Mnorm{\opS \opT z}^2 \leq \Mnorm{\opS z}^2.
    \] 
\end{proof}

\section{Omitted proofs of \cref{sec:variance}} \label{appendix:variance}
\subsection{Proofs of the lemmas in \cref{sec:variance}}\label{appendix:variance.1}
\begin{proof} [Proof of \cref{lem:e.1}]
Since the inequality holds if  $\bv=0$, let's assume that $\bv\neq 0$.

Define $x$ as $y-z$, $u$ as $\opS y - \opS z$. Since $u \in C_\delta$, there exist $\phi \in \left[ \delta, {\pi}/{2} \right]$ such that 
\[
   \left< \bv,u\right>_M= \sin\phi \Mnorm{\bv} \Mnorm{u}.
\]
Due to cocoersivity of $\opS$, $x$ and $u$ must satisfy 
\[
 \left< x,u \right>_M = \left< y-z, \opS y-\opS z \right>_M \geq 0.
\]
Since $\bv$ is nonzero, decompose $x$ and $u$ as
\[
x = \left<x,\bv\right>_M\frac{\bv}{\Mnorm{\bv}^2} + \bv_x^\perp, \quad u = \left<u,\bv\right>_M\frac{\bv}{\Mnorm{\bv}^2} + \bv_u^\perp.
\]
Both $\bv_x^\perp$ and $\bv_u^\perp$ are orthogonal to $\bv$ and 
\[
\Mnorm{\bv_u^\perp} = \cos \phi \Mnorm{u},
\quad \Mnorm{\bv_x^\perp}^2 =  \Mnorm{x}^2 - \left(\frac{\left< x,\bv\right>_M}{ \Mnorm{\bv}}\right)^2.
\]Compute $\left< x,u \right>$ using the decomposition above,
\begin{equation*}
\begin{split}
\left< x,u \right>_M
&=  \frac{1}{\Mnorm{\bv}^2} \left<x,\bv\right>_M\left<u,\bv\right>_M+ \left< \bv_x^\perp , u \right>_M\\
&=  \frac{1}{\Mnorm{\bv}} \left<x,\bv\right>_M  \sin\phi \Mnorm{u}+ \left< \bv_x^\perp , \bv_u^\perp \right>_M\\
&\leq \frac{1}{\Mnorm{\bv}} \left<x,\bv\right>_M  \sin\phi \Mnorm{u}+ \cos \phi \Mnorm{u} \sqrt{ \Mnorm{x}^2 - \left(\frac{\left< x,\bv\right>}{\Mnorm{\bv}}\right)^2 }.
\end{split}
\end{equation*}

If $\left< x,-\bv\right>_M\leq 0$, then $\left< -\bv,  y- z \right>_M\leq 0 \leq \cos\delta \Mnorm{\bv} \Mnorm{ y- z}$ and the conclusion holds. Thus, consider only the case where $\left< x,-\bv\right>_M> 0$. In such case,  $ \left< x,u \right>_M \geq 0$ with $\Mnorm{u}>0$ gives
\[
0< \frac{1}{\Mnorm{\bv}} \left<x,-\bv\right>_M  \sin\phi  \leq \cos \phi  \sqrt{ \Mnorm{x}^2 - \left(\frac{\left< x,-\bv\right>_M}{\Mnorm{\bv}}\right)^2 },
\]
which leads to a conclusion by squaring each sides :
\[
\left< x,-\bv\right>_M\leq \cos\phi \Mnorm{\bv} \Mnorm{ x}\leq \cos\delta \Mnorm{\bv} \Mnorm{ x}.
\]
\end{proof}

\begin{proof} [Proof of \cref{lem:lim_orthogonal}]
Choose a point $z$ in $\cH$. To prove by contradiction, suppose that for any $l$, there exists $k_l>l$ such that 
\[
\opS y^{k_l} \in \opS z+ C_\delta, \quad \opS y^{k_l} \neq \opS z.
\]
The subsequence $y^{k_1}, y^{k_2}, y^{k_3}, \dots$ satisfies the inequality below for all $l$, due to \cref{lem:e.1}. 
\[
 \left< -\bv,  y^{k_l}- z \right>_M\leq \cos\delta \Mnorm{\bv} \Mnorm{ y^{k_l}- z}.
\]
Divide each side by $k_l$ and take a limit as $l\to \infty$. Since $\lim_{l\to \infty} {y^{k_l}}/{ k_l} = -\gamma \bv$ strongly,
\[
 \gamma \Mnorm{\bv}^2 = \left< -\bv,  -\gamma \bv \right>_M\leq \cos\delta \Mnorm{\bv} \Mnorm{ -\gamma \bv} < \gamma \Mnorm{\bv}^2,
\]
which yields a contradiction. 

Thus, when $z$ is given, for any $\delta \in \left(0, {\pi}/{2}\right)$, there exist a $N_{\delta, z}$ such that for all $k>N_{\delta, z}$, it is either $\opS y^k = \opS z$ or $\opS y^k \notin \opS z + C_\delta$. As a conclusion, for all $k>N_{\delta, z}$,
\[
\left<\bv, \opS y^k - \opS z \right>_M\leq  \sin\delta \Mnorm{\bv} \Mnorm{\opS y^k - \opS z}.
\]
\end{proof}

\subsection{Detailed computation of \cref{eq:5.1}} \label{appendix:variance.2}
\begin{proof} [Proof of {\cref{eq:5.1}}]
    
An $({1}/{2}) \mathrm{-averaged}$ operator $\opT$ in $\reals^2$ is defined as,
\[
\opT\colon (x,y) \mapsto \left( x-\frac{1+x-y}{2} , y-\frac{1+y-x}{2} \right),
\]
with the infimal displacement vector $\bv$ of $\range{\opI-\opT}$ as $\left({1}/{2}, {1}/{2}\right)$. The {RC-FPI} by $\opT$ with the distribution as a uniform distribution on $\left\{ (1,0), (0,1) \right\}$.  The random coordinate operators are respectively,

\[
\opT_{(1,0)}: (x,y) \mapsto \left( x-\frac{1+x-y}{2} , y \right),
\]
\[
\opT_{(0,1)}: (x,y) \mapsto \left( x, y-\frac{1+y-x}{2} \right).
\]

When we set the initial point $\left( x^0, y^0\right)$ as the origin, from the relations 
\begin{equation*}
\begin{split}
&\expec\left[x^ {k+1}\right] =  \expec\left[ x^k\right] -\frac{1}{4} - \frac{1}{4}\expec\left[ x^k-y^k\right] , \\
&\expec\left[y^ {k+1}\right] =  \expec\left[y^k\right] -\frac{1}{4} + \frac{1}{4}\expec\left[ x^k-y^k\right] , \\
&\expec\left[x^ {k+1}-y^{k+1}\right] =  \frac{1}{2}\expec\left[ x^k-y^k\right] ,
\end{split}
\end{equation*}
each expectations have a value of $\expec\left[x^ {k}\right] = \expec\left[y^ {k}\right]=- k/4$. 

Next, an expectation $\expec\left[\left({ x^k - y^ {k}}\right)^2\right]$ has a recurrence relation of
\begin{equation*}
\begin{split}
\expec\left[\left({x^ {k+1} - y^ {k+1}}\right)^2\right] 
&= \expec_{\left(  x^k, y^k\right)}
\expec\left[\left({x^ {k+1} - y^ {k+1}}\right)^2 \mid  \left(  x^k, y^k\right) \right] \\
&= \expec_{\left(  x^k, y^k\right)}
\left[  \frac{1}{4} \left(  x^k - y^k \right)^2 + \frac{1}{4}  \right] ,\\
\end{split}
\end{equation*}
which obtains $\expec \left[ \left({  x^k - y^k }\right)^2 \right] = \left( 1-4^{-k}\right) / 3$ as a solution. 

Finally, an expectation $\expec\left[\left({ x^k}\right)^2 + \left({y^k}\right)^2 \right]$ has a relation
\begin{equation*}
\begin{split}
\expec\left[\left({x^ {k+1}}\right)^2 + \left({y^ {k+1}}\right)^2 \right] 
&=\expec_{\left(  x^k, y^k\right)} 
\left[\frac{1}{2} 
\left( \norm{\opT_{(1,0)} \left(  x^k, y^k\right) }^2 + \norm{\opT_{(0,1)} {\left(  x^k, y^k\right) }}^2 \right) 
\right]\\
&= \expec\left[\left({ x^k}\right)^2 + \left({y^k}\right)^2 \right] + \frac{1}{4} -\frac{1}{2} \expec\left[{ x^k + y^ {k}}\right] -  \frac{1}{4} \expec\left[\left({ x^k - y^ {k}}\right)^2\right],
\end{split}
\end{equation*}
which can be applied inductively, and as a result, 
\[
\expec\left[\norm{ \left(x^k, y^k\right)}^2  \right] = \frac{1}{8}k^2 + \frac{1}{24}k + \frac{1}{9}\left( 1- 4^{-k}\right).
\]

From above computations, a variance of $\left(  x^k, y^k\right)$ can be estimated explicitly as 
\[
\mvar\left(  x^k, y^k\right) =  \frac{1}{24}k + \frac{1}{9}\left( 1- 4^{-k}\right).
\]
Thus, $\limsup_{k\to\infty} k\mvar\left( (x^k, y^k) / k \right) $ is
\[
\limsup_{k\to\infty} k\mvar\left(\frac{\left(  x^k, y^k \right) } {k}\right) = \frac{1}{24},
\]
\end{proof}

\section{Omitted proofs of \cref{sec:decent_opt}}
\subsection{Proofs of the lemmas in \cref{subsec:PGEXTRA}} \label{appendix:pgextra}
\begin{proof} [Proof of \cref{lem:v_pgextra}]

Recall that the {PG-EXTRA} originated from {Condat-V\~u}, an {FPI} with  
\begin{equation*}
\opT \begin{bmatrix}
    \vx \\
    \vu
\end{bmatrix}
=
\begin{bmatrix}
    \prox_{\tau  f} (W\vx - \tau  U \vu ) \\
     \vu + \frac{1}{\tau } U \vx
\end{bmatrix}
\end{equation*}
which is a non-expansive mapping in $M$-norm, where
\begin{equation*}
    \begin{split}
        M
        = 
        \begin{bmatrix}
        \frac{1}{\tau } I & U \\
        U & \tau  I
        \end{bmatrix}.
    \end{split}
\end{equation*}

Finding the infimal displacement vector of $\opT$ is equivalent to 
\[
\argmin_{\vx, \vu} \Mnorm{
\begin{bmatrix}
    \Delta \vx \\
    \Delta \vu
\end{bmatrix}
}^2,\quad
\begin{bmatrix}
    \Delta \vx \\
    \Delta \vu
\end{bmatrix}
= \left( \opI - \opT \right)
\begin{bmatrix}
     \vx \\
     \vu
\end{bmatrix}
=
\begin{bmatrix}
    \vx - \prox_{\tau  f} (W\vx - \tau  U \vu ) \\
    - \frac{1}{\tau } U \vx
\end{bmatrix}.
\]
From $\Delta \vu =  - \frac{1}{\tau } U \vx$,
\begin{equation*}
    \begin{split}
    \Mnorm{
        \begin{bmatrix}
    \Delta \vx \\
    \Delta \vu
\end{bmatrix}
}^2 
&= \frac{1}{\tau }\norm{\Delta \vx}^2 + \tau  \norm{\Delta \vu}^2 + 2 \tr \left( \Delta \vx^T U \Delta \vu \right) \\
&= \frac{1}{\tau }\norm{\Delta \vx}^2 + \frac{1}{\tau } \norm{U \vx}^2 - \frac{2}{\tau } \tr \left( \Delta \vx^T U^2 \vx \right) \\
&= \frac{1}{\tau } \left[ \norm{\Delta \vx}^2 + \frac{1}{2}\tr \left( \vx^T (I-W) \vx \right) - \tr \left( \Delta \vx^T (I-W) \vx \right) \right]. \\
    \end{split}
\end{equation*}
When $\Delta \vx = \vx_\vC + \vx_\perp$, where $\vx_\vC = \mathbf{1}\Tilde{x}^T$ for some $\Tilde{x}\in \RR^d$ and $\mathbf{1}^T\vx_\perp = \mathbf{0}$, we have
\begin{equation*}
    \begin{split}
        & \Delta \vx = \vx - \prox_{\tau  f} (W\vx - \tau  U \vu ) \\
        & \Leftrightarrow   \vx_\vC = \left(\vx -  \vx_\perp\right) - \prox_{\tau  f} (W\left( \vx -  \vx_\perp \right) - \left( \tau  U \vu - W\vx_\perp\right) ).
    \end{split}
\end{equation*}
Since {$\range(U) = \ker(U)^\perp = \text{span}(\mathbf{1})^\perp$ from $\ker(U) = \text{span}(\mathbf{1})$ and $U$ being symmetry,}
\[
\left\{ \tau  U \vu: \vu \in \RR^{m \times d} \right\} = \left\{ \vw \in \RR^{m \times d} : \mathbf{1}^T \vw = \mathbf{0}  \right\}, \quad \mathbf{1}^T W\vx_\perp = \mathbf{1}^T \vx_\perp = \mathbf{0},
\]
we have $\tau  U \vu - W\vx_\perp = \tau  U \Tilde{\vu}$ for some $\Tilde{\vu}$. Thus, 
\begin{align*}
    \begin{bmatrix}
    \Delta \left( \vx -  \vx_\perp \right) \\
    \Delta \Tilde{\vu}
\end{bmatrix}
= \begin{bmatrix}
     \vx_\vC \\
    - \frac{1}{\tau } U \left( \vx -  \vx_\perp \right)
\end{bmatrix},
\end{align*}
and its $M$-norm is
\begin{align*}
\Mnorm{\begin{bmatrix}
    \Delta \left( \vx -  \vx_\perp \right) \\
    \Delta \Tilde{\vu}
\end{bmatrix}}^2 = \frac{1}{\tau } \left[  \norm{\vx_\vC}^2 + 
\frac{1}{2}\tr \left( \left( \vx -  \vx_\perp \right)^T (I-W) \left( \vx -  \vx_\perp \right) \right) \right].
\end{align*}

Due to the inequality $\norm{\vx_\perp}^2 \geq \tr \left( {\vx_\perp}^T (I-W) \vx_\perp \right)$ with equality only when $\vx_\perp = \mathbf{0}$,
\begin{equation*}
    \begin{split}
    \Mnorm{
        \begin{bmatrix}
    \Delta \vx \\
    \Delta \vu
\end{bmatrix}
}^2 
&= \frac{1}{\tau } \left[ \norm{\Delta \vx}^2 + \frac{1}{2}\tr \left( \vx^T (I-W) \vx \right) - \tr \left( \Delta \vx^T (I-W) \vx \right) \right] \\
&= \frac{1}{\tau } \left[ \norm{\vx_\vC}^2 + \norm{\vx_\perp}^2 + \frac{1}{2}\tr \left(  \vx ^T (I-W) \vx \right) - \tr \left(  {\vx_\perp}^T (I-W) \vx \right) \right] \\
& \geq \Mnorm{
\begin{bmatrix}
    \Delta \left( \vx -  \vx_\perp \right) \\
    \Delta \Tilde{\vu}
\end{bmatrix}
}^2 + \frac{1}{2\tau}\norm{\vx_\perp}^2\\
&\geq \Mnorm{
\begin{bmatrix}
    \Delta \left( \vx -  \vx_\perp \right) \\
    \Delta \Tilde{\vu}
\end{bmatrix}
}^2,
    \end{split}
\end{equation*}
with equality only when $\vx_\perp = \mathbf{0}$. Thus, the infimal displacement vector $\Tilde{\bv}$ of {Condat-V\~u} follows a form of 
\[
\Tilde{\bv} = \begin{bmatrix}
    \vv_x \\
    \vv_u
\end{bmatrix},
\quad \vv_x = \mathbf{1} \Tilde{x}^T,
\]
for some $\Tilde{x} \in \RR^d$. Now we may consider only the case where $\Delta \vx = \mathbf{1} x^T$. In this case, 
\[
\Mnorm{
        \begin{bmatrix}
    \Delta \vx \\
    \Delta \vu
\end{bmatrix}
}^2 
=\frac{1}{\tau } \left[ \norm{\Delta \vx}^2 + \frac{1}{2}\tr \left( \vx^T (I-W) \vx \right) \right],
\]
and the relation
\[
\mathbf{1} x^T = \vx - \prox_{\tau  f} (W\vx - \tau  U \vu ) 
\]
must hold. This relation is equivalent to 
\[
0\in \tau  \partial f(\vx - \mathbf{1} x^T) + \vx - \mathbf{1} x^T - W\vx + \tau  U \vu.
\]

By taking direction of $\mathbf{1}$  to consideration,
\[
0\in \tau  \mathbf{1}^T \partial f(\vx - \mathbf{1} x^T) + \mathbf{1}^T \vx - m x^T -\mathbf{1}^T \vx.
\]
When we set the new variable $\vy = \vx-\mathbf{1}x^T$, $x^T$ is expressed as
\[
x^T \in \tau  \frac{1}{m}\mathbf{1}^T \partial f(\vy), 
\]
which makes for some $\vg \in \partial f(\vy)$,
\begin{equation*}
    \begin{split} 
\Mnorm{
        \begin{bmatrix}
    \Delta \vx \\
    \Delta \vu
\end{bmatrix}
}^2 
&=\frac{1}{\tau } \left[ \norm{\Delta \vx}^2 + \frac{1}{2}\tr \left( \vy^T (I-W) \vy \right) \right] \\
&=\frac{1}{\tau } \left[ \tau ^2 \frac{1}{m} \norm{\mathbf{1}^T \vg}^2 + \frac{1}{2}\tr \left( \vy^T (I-W) \vy \right) \right] .
    \end{split}
\end{equation*}

Thus, the infimal displacement vector of {Condat-V\~u} is
\[
\Tilde{\bv} = 
\begin{bmatrix}
     \tau   \frac{1}{m}\mathbf{1} \mathbf{1}^T \vg \\
    -\frac{1}{\tau } U \vy
\end{bmatrix},
\]
where $\vy$ and $\vg$ are
\[
 \underset{
 \begin{subarray}{c}
  \vy \in \RR^{m \times d}\\
  \vg \in \partial f(\vy)
  \end{subarray}
 }{\argmin} ~   \left[ \tau ^2 \frac{1}{m} \norm{\mathbf{1}^T \vg}^2 + \frac{1}{2}\tr \left( \vy^T (I-W) \vy \right) \right].
\]

{Furthermore, from $\sum_{i=1}^m W_{i,j} = \sum_{j=1}^m W_{i,j} = 1$,
\begin{align*}
\tr(\vy^T (I-W) \vy) &= \sum_{i=1}^m \norm{y_i}^2 - \sum_{i,j=1}^m W_{i,j} \left< y_i, y_j\right> \\
&= \sum_{i=1}^m \norm{y_i}^2 + \frac{1}{2}\sum_{i,j=1}^m W_{i,j} \left(\norm{y_i - y_j}^2 - \norm{y_i}^2 - \norm{y_j}^2\right)\\
&= \frac{1}{2}\sum_{i,j=1}^m W_{i,j} ||y_i - y_j||^2.
\end{align*}}

As a conclusion, {with $\vw^k = \tau U \vu^k$, }the infimal displacement vector of \eqref{iter:PG-EXTRA} is,
 \[
\bv_i = 
\begin{bmatrix}
    \frac{\tau }{m} \sum_{j=1}^m g_j \\
    -\frac{1}{2} \left( y_i - \sum_{j=1}^m W_{ij}y_j \right)
\end{bmatrix}
\]
for $i=1,\dots,m$, where $\left( y_1, y_2, \dots , y_m\right)$ and $\left( g_1, g_2, \dots , g_m\right)$ are
\[
 \underset{
 \begin{subarray}{c}
 y_1, y_2, \dots y_m \in \RR^{d} \\
  g_j \in \partial f_j(y_j), 1\leq j \leq m
  \end{subarray}
  }{\argmin} ~    \frac{\tau^2}{m}\norm{ \sum_{j=1}^m g_j}^2 + \frac{1}{4} \sum_{i,j=1}^m W_{ij} \norm{y_i - y_j}^2.
\]
\end{proof}

\subsection{Detailed computation of \cref{sec:exp}} \label{appendix:infdistvec}

\begin{proof} [Calculation of the infimal displacement vector]
    From \cref{lem:v_pgextra}, the infimal displacement vector of \eqref{iter:PG-EXTRA} is
\[
\bv_i = 
\begin{bmatrix}
    \frac{\tau }{m} \sum_{j=1}^m g_j \\
    -\frac{1}{2} \left( y_i - \sum_{j=1}^m W_{ij}y_j \right)
\end{bmatrix},
\]
where $\left( y_1, y_2, \dots , y_m\right)$ and $\left( g_1, g_2, \dots , g_m\right)$ are
\[
 \underset{
 \begin{subarray}{c}
 y_1, y_2, \dots y_m \in \RR^{d} \\
  g_j \in \partial f_j(y_j), 1\leq j \leq m
  \end{subarray}
  }{\argmin} ~   \frac{\tau^2}{m} \norm{ \sum_{j=1}^m g_j}^2 + \frac{1}{4} \sum_{i,j=1}^m W_{ij} \norm{y_i - y_j}^2.
\]
Note that the subgradient of the indicator function is the normal cone operator
\[
\partial \delta_C (x) = \opN_C = \begin{cases}
    \emptyset & \text{if $z\notin C$} \\
    \left\{ y \mid \left< y,z-x\right> \leq 0, \forall z \in C\right\}& \text{if $z\in C$}.
\end{cases}
\]
Thus, with the choice $g_j = 0$, the problem of $\left( y_1, y_2, \dots , y_m\right)$ is equivalent to
\[
 \left( y_1, y_2, \dots , y_m\right) = 
 \underset{y_1 \in C_1, y_2, \dots y_m \in C_2}{\argmin} ~ \sum_{i,j=1}^m W_{ij} \norm{y_i - y_j}^2.
\]

Since
\begin{equation*}
    \begin{split}
        &\sum_{i,j=1}^m W_{ij} \norm{y_i - y_j}^2 \\
        &\hspace{0.5cm}= \frac{1}{m-1 + \epsilon } \norm{y_1 - y_2}^2 
        +\sum_{j>2}^m \frac{1}{m-1 + \epsilon } \norm{y_2 - y_j}^2 
        +\sum_{i,j>2, i\neq j}^m \frac{1}{m-2 + \epsilon } \norm{y_i - y_j}^2,
    \end{split}
\end{equation*}
$\left( y_1, y_2, \dots , y_m\right)$ take value of $y_2 = y_3 =\dots = y_m$ with
\[
\left( y_1, y_2\right) = \underset{y_1 \in C_1, y_2 \in C_2}{\argmin} ~  \norm{y_1 - y_2}^2.
\]
Now chose $g_j=0$ for each $j$,
\[
\bv_i = 
\begin{bmatrix}
    \mathbf{0} \\
    -\frac{1}{2} \left( y_i - \sum_{j=1}^m W_{ij}y_j \right)
\end{bmatrix}
= 
\begin{bmatrix}
    \mathbf{0} \\
    -\frac{1}{2} \sum_{j\neq i}^m W_{ij}\left( y_i - y_j \right)
\end{bmatrix}.
\]
With  $y_2 = y_3 =\dots = y_m$,
\[
\bv_1 = 
\begin{bmatrix}
    \mathbf{0} \\
     \frac{1}{2\left(m-1 + \epsilon\right) } \left( y_1 - y_2 \right)
\end{bmatrix}, \quad
\bv_2 = 
\begin{bmatrix}
    \mathbf{0} \\
    \frac{1}{2\left(m-1 + \epsilon \right) } \left( y_2 - y_1 \right)
\end{bmatrix}, \quad
\bv_i = 
\begin{bmatrix}
    \mathbf{0} \\
    \mathbf{0}
\end{bmatrix}, \quad i> 2.
\]

Thus, the infimal displacement vector is 
\begin{equation*}
\bv_i = \frac{b_i}{2(m-1+\epsilon)} \begin{bmatrix} \mathbf{0} \\u \end{bmatrix}, \quad 
b_i = \begin{cases}
  1& \text{if $i=1$} \\
  -1 &  \text{if $i = 2$} \\
  0 & \text{if $i>2$},
\end{cases}
\end{equation*}
with 
\[
u=  \begin{array}{ll}
\underset{u\in \overline{ \left\{ u_1 - u_2 \middle| u_1 \in C_1, u_2 \in C_2 \right\}}}{\mbox{argmin}}
&\displaystyle{
\norm{u}}.
\end{array}
\]
\end{proof}

\bibliographystyle{./elsarticle-num-names} 
\bibliography{./main}

\end{document}